\documentclass[11pt,a4paper]{article}
\usepackage[utf8]{inputenc}
\usepackage{amsmath,amssymb,amsfonts}
\usepackage{enumerate}
\usepackage[margin=3cm]{geometry}
\usepackage{hyperref}
\hypersetup{
	colorlinks=true,        
	linkcolor=blue,         
	citecolor=blue,         
	urlcolor=blue           
}
\usepackage{graphicx}
\usepackage{mathtools}
\usepackage[ruled]{algorithm2e}
\usepackage{array}
\usepackage{subcaption}

\usepackage{amsthm}
\usepackage{thmtools}
\newtheorem{theorem}{Theorem}
\newtheorem{proposition}{Proposition}
\newtheorem{corollary}{Corollary}
\newtheorem{definition}{Definition}
\newtheorem{assumption}{Assumption}

\newtheorem{lemma}{Lemma}
\declaretheorem[style=remark,qed=$\Diamond$,Refname={Remark,Remarks}]{remark}
\declaretheorem[style=remark,qed=$\Diamond$,Refname={Example,Examples}]{example}

\DeclareMathOperator{\Id}{Id}

\DeclareMathOperator{\dom}{dom}
\DeclareMathOperator{\ran}{ran}
\DeclareMathOperator{\gra}{gra}
\DeclareMathOperator{\zer}{zer}
\DeclareMathOperator{\Fix}{Fix}
\DeclareMathOperator{\prox}{prox}
\DeclareMathOperator*{\argmin}{argmin}
\newcommand{\Hilbert}{\mathcal{H}}
\newcommand{\R}{\mathbb{R}}

\newcommand{\ball}{\mathbb{B}}
\newcommand{\setto}{\rightrightarrows}
\newcommand{\wto}{\rightharpoonup}
\newcommand*\colvec[1]{\begin{pmatrix}#1\end{pmatrix}}

\title{Strengthened Splitting Methods for Computing Resolvents}
\author{Francisco J.\ Arag\'on Artacho\thanks{Department of Mathematics,
                             University of Alicante,
                             Alicante, \textsc{Spain}.
	                         Email:~\href{mailto:francisco.aragon@ua.es}
	                         {francisco.aragon@ua.es}}
          \and
        Rub\'en Campoy\thanks{Department of Statistics and Operational Research,
                              Universitat de València, Valencia, \textsc{Spain}.
	                          Email:~\href{mailto:ruben.campoy@uv.es}
	                          {ruben.campoy@uv.es}}
          \and
        Matthew K.\ Tam\thanks{School of Mathematics \& Statistics,
	                         The University of Melbourne,
	                         Parkville VIC 3010, \textsc{Australia}.
	                         Email:~\href{mailto:matthew.tam@unimelb.edu.au}
	                         {matthew.tam@unimelb.edu.au}}}

\begin{document}
\maketitle
\begin{abstract}
In this work, we develop a systematic framework for computing the resolvent of the sum of two or more monotone operators which only activates each operator in the sum individually. The key tool in the development of this framework is the notion of the ``strengthening'' of a set-valued operator, which can be viewed as a type of regularisation that preserves computational tractability. After deriving a number of iterative schemes through this framework, we demonstrate their application to best approximation problems, image denoising and elliptic PDEs.
\end{abstract}


\section{Introduction}
The \emph{resolvent} of a monotone operator, as studied by Minty \cite{minty1962monotone} and others \cite{BMW20,bauschke2010general,eckstein1992douglas,rockafellar1976monotone}, is an integral building block of many iterative algorithms. By specialising to the subdifferentials of convex functions, the notion includes the commonly encountered \emph{proximity operator} in the sense of Moreau~\cite{moreau1962fonctions} as a special case, as well as the \emph{nearest point/orthogonal projector} onto a set. In general, evaluating the resolvent at a point is not a straightforward task and involves solving a non-trivial monotone inclusion. Fortunately however, a number of special cases of practice importance have closed form expressions which lend themselves to efficient evaluation, such as the \emph{soft-thresholding operator} \cite[Example~4.1]{lauster2018symbolic}, which arises from $\ell_1$-regularisation, the proximity operator of the \emph{logarithmic barrier} used in interior point approaches \cite{bertocchi2020deep}, as well as many nearest point projectors onto convex sets \cite{aragon2020ORclassroom}. For an extensive list of such examples, the reader is referred to~\cite{chierchia,parikh2014proximal}.

While it is often possible to decompose a monotone operator into a sum of simpler monotone operators, each having easy to compute resolvents, this additive structure does not generally ensure ease of computing its resolvent, except in certain restrictive settings (such as \cite[Proposition~23.32]{bauschke2017}) which have limited applicability. A concrete example requiring computation of the resolvents of a sum arises in fractional programming problems as we show next.
\begin{example}[Fractional programming]
Let $\Hilbert$ be a real Hilbert space. Consider the \emph{fractional programming problem}
\begin{equation}\label{eq:fp}
 \bar{\theta} := \inf_{x\in S}\frac{f(x)}{g(x)},
\end{equation}
where $S\subseteq\Hilbert$ is nonempty, closed and convex, $f\colon\Hilbert\to{]-\infty,+\infty]}$ is proper, lower semicontinous and convex with $f(x)\geq 0$ for all $x\in S$ and  $S\cap\dom f\neq\emptyset$, and $g\colon\Hilbert\to\mathbb{R}$ is either convex or concave, differentiable and satisfies $g(S)\subseteq{]0,M]}$ for some $M>0$. To solve this problem,  a proximal-gradient-type algorithm of the following form was proposed in~\cite{bot2017proximal} (see also \cite{bot2020extrapolated}).\smallskip

\begin{algorithm}[H]
 \textbf{Initialisation.}~Choose $x_0\in S\cap\dom f$ and set $\theta_0:=\frac{f(x_0)}{g(x_0)}$\;
 \For{$k\geq 1$}{
   1.~Choose $\eta_k>0$ according to a specified rule (see \cite[Algorithms~6~\&~9]{bot2017proximal})\;
   2.~Compute $x_k$ by solving
     \begin{equation}\label{eq:fp-step2}
     x_k:=\argmin_{x\in S}\left[f(x)+\frac{1}{2\eta_k}\bigl\|x-\left(x_{k-1}-\theta_k\eta_k\nabla g(x_{k-1})\right)\bigr\|^2\right];
     \end{equation}
   3.~Set $\theta_k:=\frac{f(x_k)}{g(x_k)}$\;
 }
 \caption{Proximal-gradient-type method \cite{bot2017proximal} for \eqref{eq:fp}.\label{alg:fp}}
\end{algorithm}\smallskip

Note that the subproblem \eqref{eq:fp-step2} in Step~2 of Algorithm~\ref{alg:fp} amounts to evaluating the proximity operator of $\eta_k(f+\iota_S)$ at the point $q:=x_{k-1}-\theta_k\eta_k\nabla g(x_{k-1})$ where $\iota_S$ denotes the indicator function of the set $S$. If the subdifferential sum rule holds for $f$ and $\iota_S$, then $\partial(f+\iota_S)=\partial f+N_S$, where $\partial f$ denotes the subdifferential of a function $f$ and $N_S$ denotes the normal cone to $S$. In this case, evaluating the aforementioned proximity operator requires the computation of the resolvent of the sum $\eta_k(\partial f+N_S)$ at $q$, where both $A:=\partial f$ and $B:=\partial\iota_S=N_S$ are set-valued maximally monotone operators.
\end{example}

To overcome these difficulties, it is natural to consider iterative algorithms for computing the resolvent of the sum of monotone operators whose iteration uses the resolvents of the individual monotone operators. To this end, Combettes~\cite{combettes2009iterative} considered algorithms based on the Douglas--Rachford method and a Dykstra-type method in the product space (in the sense of Pierra \cite{pierra1984decomposition}). More recently, Arag\'on and Campoy~\cite{aragon2019computing} developed a method based on a modification of the Douglas--Rachford method for two operators which does not necessarily require a product space, which was further studied in \cite{alwadani2018asymptotic,dao2019resolvent}. Building on the work of Moudafi \cite{moudafi2014computing}, Chen and Tang \cite{chen2019iterative} devised an algorithm for computing the resolvent of sum of a monotone and a composite operator. A different approach based on a composition formula involving generalised resolvents has also been considered in works by Adly and Bourdin \cite{adly2019decomposition}, and Adly, Bourdin and Caubet \cite{adly2019proximity}.

In most of the aforementioned works, the main focus of the analysis has been on specific algorithms. In this work, our approach is different. We instead focus on the interplay between properties of the operators themselves and the underlying problem formulations. By doing so, we unify many of the existing algorithms for computing resolvents in the literature within a framework based on \emph{strengthenings} of monotone operators. This notion can be viewed as a regularisation of an operator which preserves certain computationally favourable properties. Moreover, this framework has the advantage of providing transparency and insight into the mechanism of existing methods as well as providing a way to systematically develop new algorithms.

\bigskip

The remainder of this work is structured as follows. We begin in Section~\ref{s:preliminaries} by recalling preliminaries for use in the sequel. In Section~\ref{s:strengthening}, we introduce and study the notion of \emph{strengthenings} of set-valued operators, including establishing relationships between the resolvents, continuity properties and zeros of operators and their strengthenings. Next we turn to iterative algorithms, with Section~\ref{s:fb methods} focusing on methods which incorporate forward steps and Section~\ref{s:resolvent iterations} focusing on methods which incorporate only backward steps. Finally, in Section~\ref{s:applications}, we apply our results to devise algorithms for three different applications: best approximation with three sets, ROF-type imaging denoising, and elliptic PDEs with partially blinded Laplacians. In addition, we also provide an alternative proof of Ryu's three operator splitting method \cite[Section~4]{ryu2019uniqueness} in Appendix~\ref{s:appendix}, which also covers convergence of its shadow sequence in the infinite dimensional setting.

\section{Preliminaries}\label{s:preliminaries}
Throughout this paper, $\Hilbert$ is a real Hilbert space equipped with inner product $\langle\cdot , \cdot\rangle$ and induced norm $\|\cdot\|$. We abbreviate \emph{norm convergence} of sequences in $\Hilbert$ with $\to$  and we use~$\rightharpoonup$ for \emph{weak convergence}. We denote the closed ball centered at $x\in\Hilbert$ of radius $\delta>0$ by $\ball(x,\delta)$.

\subsection{Operators}
Given a non-empty set $D\subseteq\Hilbert$, $A:D\setto\Hilbert$ denotes a \emph{set-valued operator} that maps any point from $D$ to a subset of $\Hilbert$, i.e., $A(x)\subseteq \Hilbert$ for all $x\in D$. In the case when $A$ always maps to singletons, i.e., $A(x)=\{u\}$ for all $x\in D$, $A$ is said to be a \emph{single-valued mapping} and it is denoted by $A:D\to\Hilbert$. In an abuse of notation, we may write $A(x)=u$ when $A(x)=\{u\}$. Note that one can always write $A:\Hilbert\setto\Hilbert$ by setting $A(x):=\emptyset$ for all $x\not\in D$. The \emph{domain}, the \emph{range}, the \emph{graph}, the set of \emph{fixed points} and the set of \emph{zeros} of $A$, are denoted, respectively, by $\dom A$, $\ran A$, $\gra A$, $\Fix A$ and $\zer A$;~i.e.,
\begin{gather*}
\begin{align*}
\dom A&:=\left\{x\in\Hilbert : A(x)\neq\emptyset\right\},&
\ran A&:=\left\{u\in\Hilbert : \exists x\in\Hilbert: u\in A(x) \right\},&\\
\gra A&:=\left\{(x,u)\in\Hilbert\times\Hilbert : u\in A(x)\right\},&
\Fix A&:=\left\{x\in\Hilbert : x\in A(x)\right\},
\end{align*}\\
\text{and} \quad \zer A:=\left\{x\in\Hilbert : 0\in A(x)\right\}.
\end{gather*}
The \index{identity operator}\emph{identity operator} is the mapping $\Id:\Hilbert\to\Hilbert$ that maps every point to itself. The \emph{inverse operator} of $A$, denoted by $A^{-1}$, is defined through
$x\in A^{-1}(u) \iff u\in A(x).$
\begin{definition}[$\alpha$-monotonicity]\label{def:alpha monotone}
Let $\alpha\in\mathbb{R}$. An operator $A:\Hilbert\setto\Hilbert$ is \emph{$\alpha$-monotone} if
$$ \langle x-y,u-v\rangle \geq \alpha\|x-y\|^2\quad\forall(x,u),(y,v)\in\gra A. $$
Furthermore, an $\alpha$-monotone operator $A$ is said to be \emph{maximally $\alpha$-monotone} if there exists no $\alpha$-monotone operator $B\colon\Hilbert\setto\Hilbert$ such that $\gra B$ properly contains $\gra A$.
\end{definition}
Depending on the sign of $\alpha$, Definition~\ref{def:alpha monotone} captures three important classes of operators in the literature. Firstly, an operator is monotone (in the classical sense) if it is $0$-monotone. Secondly, an operator is $\alpha$-strongly monotone (in the classical sense) if it is $\alpha$-monotone for $\alpha>0$. And, finally, an operator is $\alpha$-weakly monotone (or $\alpha$-hypomonotone) if it is $\alpha$-monotone for $\alpha<0$.

\begin{definition}[Resolvent operator]
Given an operator $A\colon\Hilbert\setto\Hilbert$, the \emph{resolvent} of $A$ with parameter $\gamma>0$ is the operator $J_{\gamma A}\colon\Hilbert\setto\Hilbert$ defined by $J_{\gamma A}:=(\Id+\gamma A)^{-1}$.
\end{definition}

\begin{proposition}[Resolvents of $\alpha$-monotone operators]\label{prop:Jalpha}
Let $A:\Hilbert\setto\Hilbert$ be $\alpha$-monotone and let $\gamma>0$ such that $1+\gamma\alpha>0$. Then
\begin{enumerate}[(i)]
\item $J_{\gamma A}$ is single-valued,
\item $\dom J_{\gamma A}=\Hilbert$ if and only if $A$ is maximally $\alpha$-monotone.
\end{enumerate}
\end{proposition}
\begin{proof}
See~\cite[Proposition 3.4]{dao2019adaptive}.
\end{proof}

\begin{definition}
	Let $D$ be a nonempty subset of $\Hilbert$ and let $T:D\to\Hilbert$. The operator $T$ is said to be
	\begin{enumerate}[(i)]
		\item \emph{$\kappa$-Lipschitz continuous} for $\kappa>0$ if
		\begin{equation*}
		\|T(x)-T(y)\|\leq \kappa \|x-y\|\quad\forall x,y\in D;
		\end{equation*}
		\item \emph{locally Lipschitz continuous} if, for all $x_0\in D$, there exist
		$\delta,\kappa>0$ such that
		\begin{equation*}
		\|T(y)-T(z)\|\leq \kappa \|y-z\|\quad\forall z,y\in \mathbb{B}(x_0,\delta)\cap D;
		\end{equation*}
		\item \emph{nonexpansive} if it is Lipschitz continuous with constant $1$;
 		\item $\alpha$\emph{-averaged} for  $\alpha\in\,]0,1[$ if there exists a nonexpansive operator $R:D\mapsto\Hilbert$ such that
 		\begin{equation*}
 		T=(1-\alpha)I+\alpha R;
 		\end{equation*}
 		\item $\alpha$\emph{-negatively averaged} for $\alpha\in{]0,1[}$ if $-T$ is $\alpha$-averaged;
		\item $\beta$-\emph{cocoercive} for $\beta>0$ if
		\begin{equation*}
		\langle x-y,T(x)-T(y)\rangle \geq \beta\|T(x)-T(y)\|^2\quad\forall x,y\in D.
		\end{equation*}
	\end{enumerate}
\end{definition}

\begin{remark}[Lipschitz continuity versus cocoercivity]
By the Cauchy--Schwarz inequality, any $\beta$-cocoercive mapping is $\frac{1}{\beta}$-Lipschitz continuous. In general, cocoercivity of an operator is a stronger condition than Lipschitz continuity, except when the operator is the gradient of a differentiable convex function: in this case the Baillon--Haddad theorem states that both notions are equivalent (see, e.g.,~\cite[Corollary~18.17]{bauschke2017}).
\end{remark}

\subsection{Functions and Subdifferentials}
An extended real-valued function ${f:\Hilbert\to]-\infty,+\infty]}$ is said to be \emph{proper} if its \emph{domain}, $\dom f:=\{x\in\Hilbert : f(x)<+\infty\}$, is nonempty. We say $f$ is \emph{lower semicontinuous (lsc)} if, at any $\bar{x}\in\Hilbert$,
$$f(\bar x)\leq \liminf_{x\to\bar x} f(x).$$
A function $f$ is said to be $\alpha$-convex, for $\alpha\in\R$, if $f-\frac{\alpha}{2}\|\cdot\|^2$ is convex; i.e, for all $x,y\in\Hilbert$,
\begin{equation*}
f((1-\lambda)x+\lambda y)\leq \lambda f(x)+(1-\lambda)f(y)-\frac{\alpha}{2}\lambda(1-\lambda)\|x-y\|^2, \quad\forall \lambda\in[0,1].
\end{equation*}
Clearly, $0$-convexity coincides with classical convexity. We say $f$ is \emph{strongly convex} when $\alpha>0$ and \emph{weakly convex} (or \emph{hypoconvex}) when $\alpha<0$.

For any extended real valued function $f$, the \emph{Fenchel conjugate} of $f$ is denoted by $f^*(\phi):=\sup_{x\in\Hilbert}\{\langle x,\phi\rangle -f(x)\}$ for all $\phi\in\Hilbert$. The \emph{Fr\'echet subdifferential} of $f$ at $x\in\dom f$ is given by
\begin{equation*}
\partial f(x):=\left\{u\in\Hilbert : \liminf_{\substack{y\to x\\ y\neq x}} \frac{f(y)-f(x)-\langle u,y-x\rangle}{\|y-x\|}\geq 0\right\}
\end{equation*}
and $\partial f(x):=\emptyset$ at $x\not\in\dom f$. When $f$ is differentiable at $x\in\dom f$, then $\partial f(x)=\{\nabla f(x)\}$ where $\nabla f(x)$ denotes the gradient of $f$ at $x$. When $f$ is convex, then $\partial f(x)$ coincides with the classical \emph{(convex) subdifferential} of $f$ at $x$, which is the set
\begin{equation*}
\{u\in\Hilbert: f(x)+\langle u,y-x\rangle \leq f(y),\,\forall y\in\Hilbert\}.\\
\end{equation*}
Given a nonempty set $C\subseteq\Hilbert$, the \emph{indicator function} of $C$, $\iota_C:\Hilbert\to{]-\infty,+\infty]}$, is defined as $$\iota_C(x):=\left\{\begin{array}{cl}
0, & \text{if } x\in C,\\
+\infty, & \text{if } x\not\in C.\end{array}\right.$$
When $C$ is a convex set, $\iota_C$ is a convex function whose subdifferential becomes the \emph{(convex) normal cone} to $C$, $N_C:\Hilbert\setto\Hilbert$, given by
\begin{equation*}
\partial \iota_C(x)=N_C(x):=\left\{\begin{array}{ll}\{u\in\Hilbert : \langle u, c-x \rangle\leq 0, \, \forall c\in C \}, &\text{if }x\in C,\\
\emptyset, & \text{otherwise.}\end{array}\right.
\end{equation*}

\begin{example}[Monotonicity of subdifferentials and normal cones]\label{ex:proxproj}
The subdifferential and the normal cone are well-known examples of maximally monotone operators.
\begin{enumerate}[(i)]
\item Let $f:\Hilbert\to{]-\infty,+\infty]}$ be proper, lsc and $\alpha$-convex for $\alpha\in\R$. Then, $\partial f$ is a maximally $\alpha$-monotone operator. Furthermore, given $\gamma>0$ such that $1+\gamma\alpha>0$, it holds that $J_{\gamma \partial f}=\prox_{\gamma f}$, where $\prox_{\gamma f}:\Hilbert\setto\Hilbert$ is the \emph{proximity operator} of $f$ (with~parameter~$\gamma$) defined at $x\in\Hilbert$ by
\begin{equation*}
\prox_{\gamma f}(x):=\argmin_{u\in\Hilbert} \left( f(u)+\frac{1}{2\gamma}\|x-u\|^2\right),
\end{equation*}
see, e.g.,~\cite[Lemma~5.2]{dao2019adaptive}.\label{ex:prox}
\item Let $C\subseteq\Hilbert$ be a nonempty, closed and convex set. Then, the normal cone $N_C$ is maximally monotone. Furthermore, $J_{N_C}=P_C$, where $P_C:\Hilbert\setto\Hilbert$ denotes the \emph{projector} onto $C$ defined at $x\in\Hilbert$ by
\begin{equation*}
P_C(x):=\argmin_{c\in C}  \|x-c\|,
\end{equation*}
see, e.g.,~\cite[Example~20.26 and Example~23.4]{bauschke2017}.\qedhere
\end{enumerate}
\end{example}

\section{Strengthenings of set-valued operators}\label{s:strengthening}
In this section, we introduce the notion of the \emph{strengthening} of a set-valued operator and study its properties. This idea appears without name in \cite{dao2019resolvent} which, in turn, builds on the special cases considered in \cite{aragon2019computing}. This concept was motivated by the ideas of \cite{combettes2009iterative}, where a particular case of the strengthening was employed in algorithm analysis.

\begin{definition}[$(\theta,\sigma)$-strengthening]
Let $\theta>0$ and $\sigma\in\R$. Given $A\colon\Hilbert\setto\Hilbert$, the \emph{$(\theta,\sigma)$-strengthening} of $A$ is the operator $A^{(\theta,\sigma)}\colon\Hilbert\setto\Hilbert$ defined by
\begin{equation}\label{eq:strength def}
  A^{(\theta,\sigma)}:=A\circ(\theta\Id)+\sigma\Id.
\end{equation}
\end{definition}

\begin{remark}
In \cite[Definition~3.2]{aragon2019computing}, the authors define the \emph{$\beta$-strengthening} of an operator $A$ for $\beta\in{]0,1[}$ as the operator $A^{(\beta)}:=A^{(\theta,\sigma)}$ with $\theta:=\frac{1}{\beta}$ and $\sigma:=\frac{1-\beta}{\beta}$.
\end{remark}

We recall next the concept of perturbation of an operator, which was studied
in~\cite{BHM14}. We follow the notation used in~\cite{BM17}.

\begin{definition}[Inner perturbation]
Let $A\colon\Hilbert\setto\Hilbert$ and let $w\in\Hilbert$. The \emph{inner $w$-perturbation} of $A$ is the operator $A_w\colon\Hilbert\setto\Hilbert$ defined at $x\in\Hilbert$ by
$$ A_w(x) := A(x-w). $$
\end{definition}

Of particular interest in this work, will be the strengthenings of inner perturbations. In other words, given an operator $A\colon\Hilbert\setto\Hilbert$ and a point $q\in\Hilbert$, we shall study the operator $(A_{-q})^{(\theta,\sigma)}$ which is given by
$$(A_{-q})^{(\theta,\sigma)}=A_{-q}\circ (\theta\Id)+\sigma\Id=A\circ (\theta\Id+q)+\sigma\Id.$$
\begin{remark}\label{r:gbar}
Suppose $f\colon\Hilbert\to{]-\infty,+\infty]}$ is a proper, lsc and $\alpha$-convex function for $\alpha\in\R$. Then the subdifferential sum rule ensures
 $$ ((\partial f)_{-q})^{(\theta,\sigma)} = \partial f\circ(\theta\Id+q)+\sigma\Id = \partial\left(\frac{1}{\theta}f\circ(\theta\Id+q) +\frac{\sigma}{2}\|\cdot\|^2\right).$$
In other words, the strengthening of the inner perturbation of a subdifferential coincides with the subdifferential of the proper, lsc and $(\theta\alpha+\sigma)$-convex function $x\mapsto \frac{1}{\theta}f(\theta x+q) +\frac{\sigma}{2}\|x\|^2$.
\end{remark}

The following property shows that the resolvent of a strengthening can be computed using the resolvent of the original operator, and vice versa.
\begin{proposition}\label{prop:strength_resolvent}
Let $A\colon\Hilbert\setto\Hilbert$, let $q\in\Hilbert$, let $\theta,\gamma>0$ and let $\sigma\in\R$. Then the following assertions hold.
\begin{enumerate}[(i)]
\item\label{prop:strength_resolvent_I} $A$ is (maximally) $\alpha$-monotone if and only if $(A_{-q})^{(\theta,\sigma)}$ is (maximally) $(\theta\alpha+\sigma)$-monotone.
\item\label{prop:strength_resolvent_II} If $1+\gamma\sigma\neq 0$, then
\begin{equation*}
J_{\gamma (A_{-q})^{(\theta,\sigma)}}=\frac{1}{\theta}\left(J_{\frac{\gamma\theta}{1+\gamma\sigma}A}\circ\left(\frac{\theta}{1+\gamma\sigma}\Id+q\right)-q\right).
\end{equation*}
If, in addition, $A$ is maximally $\alpha$-monotone and $1+\gamma(\theta\alpha+\sigma)>0$, then both $J_{\gamma (A_{-q})^{(\theta,\sigma)}}$ and $J_{\frac{\gamma\theta}{1+\gamma\sigma}A}$ are single-valued with full domain.
\end{enumerate}
\end{proposition}
\begin{proof}
See~\cite[Proposition~2.1]{dao2019resolvent}.
\end{proof}

Using the expression \eqref{eq:strength def}, it can be seen that the strengthening of an operator can be evaluated whenever the original operator can be evaluated. Next, we investigate how other properties are preserved under taking strengthenings.

\begin{theorem}\label{th:lip coco}
Let $B\colon\Hilbert\to\Hilbert$, let $q\in\Hilbert$, let $\theta>0$, let $\sigma\in\R$, and let $\kappa>0$. Then the following assertions hold.
\begin{enumerate}[(i)]
    \item\label{it:Lip of streng} $B$ is $\kappa$-Lipschitz continuous if and only if $(B_{-q})^{(\theta,0)}$ is $\kappa\theta$-Lipschitz continuous. Consequently, $(B_{-q})^{(\theta,\sigma)}$ is $(\kappa\theta+|\sigma|)$-Lipschitz continuous.
    \item\label{it:local Lip of streng} $B$ is locally Lipschitz if and only if $(B_{-q})^{(\theta,\sigma)}$ is locally Lipschitz.
    \item\label{it:coco of streng} $B$ is $\beta$-cocoercive if and only if $(B_{-q})^{(\theta,0)}$ is $\frac{\beta}{\theta}$-cocoercive. Consequently, if $\sigma>0$, then $(B_{-q})^{(\theta,\sigma)}$ is $\mu$-cocoercive with
\begin{equation*}
 \mu := \left(\frac{\theta}{\beta}+\sigma\right)^{-1}.
\end{equation*}
\end{enumerate}
\end{theorem}
\begin{proof}
First note that since Lipschitz continuity and cocoercivity are preserved undertaking inner perturbations, it suffices to prove the result for $q=0$.

\eqref{it:Lip of streng}:~The equivalence follows immediately from definition the $(\theta,0)$-strengthening. Using the identity $B^{(\theta,\sigma )}=B^{(\theta,0)}+\sigma \Id$, we then deduce
\begin{align*}
\|B^{(\theta,\sigma )}(x)-B^{(\theta,\sigma )}(y)\|
&\leq \|B^{(\theta,0)}(x)-B^{(\theta,0)}(y)\|+|\sigma |\|x-y\| \\
&\leq \kappa\theta\|x-y\|+|\sigma |\|x-y\|,
\end{align*}
from which the result follows.
\eqref{it:local Lip of streng}:~The proof is similar to \eqref{it:Lip of streng}.
\eqref{it:coco of streng}:~The equivalence follows immediately from definition of the $(\theta,0)$-strengthening. Next, note that
\begin{equation}\label{eq:young}\begin{aligned}
\frac{\alpha_1\alpha_2}{\alpha_1+\alpha_2}\|u+v\|^2
&\leq \frac{\alpha_1\alpha_2}{\alpha_1+\alpha_2}\left(\|u\|^2+\|v\|^2+2\|u\|\|v\|\right) \\
&\leq \frac{\alpha_1\alpha_2}{\alpha_1+\alpha_2}\left(\|u\|^2+\|v\|^2+\frac{\alpha_1}{\alpha_2}\|u\|^2+\frac{\alpha_2}{\alpha_1}\|v\|^2\right)\\
&= \alpha_1\|u\|^2 + \alpha_2\|v\|^2,
\end{aligned}\end{equation}
for all weights $\alpha_1,\alpha_2>0$ and $u,v\in\Hilbert$. Using the identity $B^{(\theta,\sigma )}=B^{(\theta,0)}+\sigma \Id$ and the {$\frac{\beta}{\theta}$-cocoercivity} of $B^{(\theta,0)}$, followed by an application of \eqref{eq:young} with weights $\alpha_1=\frac{\beta}{\theta}$ and $\alpha_2=\frac{1}{\sigma }$ we obtain
\begin{align*}
 \langle x-y,B^{(\theta,\sigma )}(x)-B^{(\theta,\sigma )}(y)\rangle
 &= \langle x-y, B^{(\theta,0)}(x)-B^{(\theta,0)}(y)\rangle + \frac{1}{\sigma }\|\sigma x-\sigma y\|^2 \\
 & \geq \frac{\beta}{\theta}\|B^{(\theta,0)}(x)-B^{(\theta,0)}(y)\|^2 + \frac{1}{\sigma }\|\sigma x-\sigma y\|^2\\
 &\geq\left(\frac{\theta}{\beta}+\sigma \right)^{-1}\|B^{(\theta,\sigma )}(x)-B^{(\theta,\sigma )}(y)\|^2,
\end{align*}
which establishes the result.
\end{proof}

The following proposition characterises the structures of the zeros of the sum of strengthenings of operators in terms of resolvents. It will be key in the development of algorithms in subsequent sections.

\begin{proposition}\label{prop:zero of strengthenings}
Let $A_i\colon\Hilbert\setto\Hilbert$ and $\sigma_i\in\mathbb{R}$ for $i\in\{1,\dots,n\}$ with $\sigma:=\sum_{i=1}^n\sigma_i>0$. Let $q\in\Hilbert$ and $\theta>0$. Then
\begin{equation}\label{eq:strengthening sum}
J_{\frac{\theta}{\sigma}\left(\sum_{i=1}^nA_i\right)}(q) = \left\{\theta x +q: x\in \zer\left(\sum_{i=1}^n((A_i)_{-q})^{(\theta,\sigma_i)}\right) \right\}.
\end{equation}
Consequently, if each $A_i$ is $\alpha_i$-monotone, $\sum_{i=1}^n(\theta\alpha_i+\sigma_i)>0$ and $q\in\ran\left(\Id+\frac{\theta}{\sigma}\sum_{i=1}^nA_i\right)$, then $J_{\frac{\theta}{\sigma}\left(\sum_{i=1}^nA_i\right)}(q)$ is a singleton and $\frac{1}{\theta}\left(J_{\frac{\theta}{\sigma}\left(\sum_{i=1}^nA_i\right)}(q)-q\right)$ is the unique element of \break $\zer\left(\sum_{i=1}^n((A_i)_{-q})^{(\theta,\sigma_i)}\right).$
\end{proposition}
\begin{proof}
For convenience, denote $A:=\sum_{i=1}^nA_i$. Then we have
\begin{align*}
  x\in\zer\left(\sum_{i=1}^n((A_i)_{-q})^{(\theta,\sigma_i)}\right) &\iff 0 \in \sum_{i=1}^n\left(A_i(\theta x+q)+\sigma_ix\right)= A(\theta x+q) + \sigma x \\
  &\iff q \in \left( \Id + \frac{\theta}{\sigma}A\right)(\theta x+q) \\
  &\iff \theta x+q \in J_{\frac{\theta}{\sigma} A}(q),
\end{align*}
which establishes \eqref{eq:strengthening sum}.
Now suppose that $A_i$ is $\alpha_i$-monotone for $i=1,\dots,n$. Then $((A_i)_{-q})^{(\theta,\sigma_i)}$ is $(\theta\alpha_i+\sigma_i)$-monotone by Proposition~\ref{prop:strength_resolvent}\eqref{prop:strength_resolvent_I} and hence $\sum_{i=1}^n((A_i)_{-q})^{(\theta,\sigma_i)}$ is $\alpha$-strongly monotone for $\alpha:=\sum_{i=1}^n(\theta\alpha_i+\sigma_i)>0$.
The result follows by noting that a strongly monotone operator has at most one zero (see, e.g.,~\cite[Proposition~23.35]{bauschke2017}) and that $q\in\ran\left(\Id+\frac{\theta}{\sigma}\sum_{i=1}^nA_i\right)$ if and only if $J_{\frac{\theta}{\sigma}\sum_{i=1}^nA_i}(q)\neq\emptyset$.
\end{proof}

\begin{remark}
Proposition~\ref{prop:zero of strengthenings} formalises and extends a number of formulations which can be found in the literature.
\begin{enumerate}[(i)]
\item If $\beta=\frac{1}{\theta}\in{]0,1[}$ and $\sigma_1=\dots=\sigma_n=\frac{1-\beta}{\beta}$, then
  $$ \zer\left(\sum_{i=1}^nA_i^{(\beta)}\right) = \beta J_{\frac{1}{n(1-\beta)}\sum_{i=1}^nA_i}(0). $$
  This recovers \cite[Proposition~4.2]{aragon2019computing}. Moreover, this result can be generalised in the following way. Let $k\in\{1,\dots,n-1\}$. If $\beta=\frac{1}{\theta}\in{]0,1[}$,
$\sigma_1=\dots=\sigma_k=\frac{1-\beta}{\beta}$ and $\sigma_{k+1}=\dots=\sigma_n=0$, then
$$ \zer\left(\sum_{i=1}^k A_i^{(\beta)}+\sum_{i=k+1}^nA_i(\cdot/\beta)\right) = \beta J_{\frac{1}{k(1-\beta)}\sum_{i=1}^nA_i}(0). $$

\item Consider weights $\omega_1,\dots,\omega_n>0$ with $\sum_{i=1}^n\omega_i=1$. If $\theta=1$ and $\sigma_1,\ldots,\sigma_n>0$ with $\sum_{i=1}^n\sigma_i=1$, then
 $$ J_{\sum_{i=1}^n\omega_iA_i}(q) = \zer\left( \sum_{i=1}^n\left((\omega_iA_i)_{-q}\right)^{(1,\sigma_i)}\right) + q = \zer\left( -q+\Id + \sum_{i=1}^n\omega_iA_i\right).$$
This fact is used in \cite[Theorem~2.8]{combettes2009iterative}.
\item Consider $n=2$, $\theta>0$, $\omega>0$, $\vartheta,\tau\in\mathbb{R}$ and $q,r,r_A,r_B\in\Hilbert$ satisfying
$$\vartheta+\tau=\frac{\theta}{\omega}\quad\text{and}\quad r_A+r_B=\frac{1}{\omega}(q+r).$$
Let $A:=A_1$ and $B:=A_2$. Take $\sigma_1=\sigma_2=\frac{\theta}{2\omega}$. Then
\begin{align*}
J_{\omega(A+B)}(r)&=\theta\zer\left(((A)_{-r})^{(\theta,\sigma_1)}+((B)_{-r})^{(\theta,\sigma_2)}\right)+r\\
&=\theta\zer\left(A\circ(\theta\Id+r)+B\circ(\theta\Id+r)+\frac{\theta}{\omega}\Id\right)+r\\
&=\theta\zer\left(A\circ(\theta\Id-q)+\vartheta\Id+B\circ(\theta\Id-q)+\tau\Id-\frac{1}{\omega}(q+r)\right)-q\\
&=\theta\zer\left(A\circ(\theta\Id-q)+\vartheta\Id-r_A+B\circ(\theta\Id-q)+\tau\Id-r_B\right)-q,
\end{align*}
which recovers \cite[Proposition~3.1]{dao2019resolvent}.\qedhere
\end{enumerate}
\end{remark}

\begin{corollary}\label{cor:zero of strengthenings prox}
For each $i\in\{1,\ldots,n\}$, let $f_i:\Hilbert\to{]-\infty,+\infty]}$ be proper, lsc and  $\alpha_i$-convex, with $\alpha_i\in\R$. Let $\sigma_1,\ldots,\sigma_n\in\R$ such that $\sigma:=\sum_{i=1}^n\sigma_i> 0$. Let $q\in\Hilbert$ and $\theta>0$, and suppose that $\sum_{i=1}^n(\theta\alpha_i+\sigma_i)>0.$
Then $q\in\ran\left(\Id+\frac{\theta}{\sigma}\sum_{i=1}^n\partial f_i\right)$ if and only if
\begin{equation}\label{eq:rangecond_prox}
q-\prox_{\frac{\theta}{\sigma}\sum_{i=1}^n f_i}(q) \in \left(\frac{\theta}{\sigma}\sum_{i=1}^n \partial f_i\right) \left(\prox_{\frac{\theta}{\sigma}\sum_{i=1}^n f_i}(q)\right).
\end{equation}
In this case, $\frac{1}{\theta}\left(\prox_{\frac{\theta}{\sigma}\sum_{i=1}^n f_i}(q)-q\right)$ is the unique element of $\zer\left(\sum_{i=1}^n((\partial f_i)_{-q})^{(\theta,\sigma_i)}\right)$.
\end{corollary}
\begin{proof}
The reverse implication is immediate. To prove the forward implication, consider $x\in\Hilbert$ such that $q\in x+\frac{\theta}{\sigma}\sum_{i=1}^n\partial f_i(x)$. Combining this with the inclusion $\sum_{i=1}^n\partial f_i(x)\subseteq \partial(\sum_{i=1}^n f_i)(x)$ gives
$$
q\in \left(\Id+\frac{\theta}{\sigma}\partial\left(\sum_{i=1}^n f_i\right)\right)(x) \iff x=J_{\frac{\theta}{\sigma}\partial\left(\sum_{i=1}^n f_i\right)}(q).
$$
Since $f_i$ is proper, lsc and $\alpha_i$-convex, then $\sum_{i=1}^nf_i$ is proper, lsc and $\alpha$-convex with $\alpha:=\sum_{i=1}^n\alpha_i$. By assumption on the parameters, it holds that $1+\alpha\frac{\theta}{\sigma}>0$. Then, by Example~\ref{ex:proxproj}\eqref{ex:prox}, we get that
$$J_{\frac{\theta}{\sigma}\partial\left(\sum_{i=1}^n f_i\right)}(q)=\prox_{\frac{\theta}{\sigma}\sum_{i=1}^n f_i}(q),$$
and, thus, \eqref{eq:rangecond_prox} holds. The last assertion follows from Proposition~\ref{prop:zero of strengthenings} combined with the fact that $\partial f_i$ is $\alpha_i$-monotone, for $i=1,\ldots,n$, by Example~\ref{ex:proxproj}\eqref{ex:prox}.
\end{proof}

\begin{corollary}\label{cor:zero of strengthenings normalcone}
Let $C_i\subseteq\Hilbert$ be a nonempty, closed and convex set and $\sigma_i\in\R$ for $i\in\{1,\ldots,n\}$ such that $\sigma:=\sum_{i=1}^n\sigma_i>0$. Let $q\in\Hilbert$ and let $\theta>0$. Then $q\in\ran\left(\Id+\sum_{i=1}^n N_{C_i}\right)$ if and only if
\begin{equation*}
q-P_{\cap_{i=1}^n C_i}(q) \in \left(\sum_{i=1}^n N_{C_i}\right) \left(P_{\cap_{i=1}^n C_i}(q)\right).
\end{equation*}
In this case, $\frac{1}{\theta}\left(P_{\cap_{i=1}^n C_i}(q)-q\right)$ is the unique element of $\zer\left(\sum_{i=1}^n((N_{C_i})_{-q})^{(\theta,\sigma_i)}\right)$.
\end{corollary}
\begin{proof}
Apply Corollary~\ref{cor:zero of strengthenings prox} with $f_i=\iota_{C_i}$, for $i\in\{1,\ldots,n\}$.
\end{proof}

\begin{remark}[Sum rule and strong CHIP]\label{r:sCHIP}
In the setting of Corollary~\ref{cor:zero of strengthenings prox}, a sufficient condition for \eqref{eq:rangecond_prox} is
\begin{equation}\label{eq:sum_rule}
\sum_{i=1}^n\partial f_i(x) = \partial\left(\sum_{i=1}^n f_i\right)(x),\quad\forall x\in\Hilbert,
\end{equation}%
which is referred to as the \emph{subdifferential sum rule}. In fact, using an argument analogous to \cite[Proposition~4.1]{aragon2018new}, it can be easily shown that the sum rule \eqref{eq:sum_rule} is equivalent to having condition \eqref{eq:rangecond_prox} to hold for all $q\in\Hilbert$. Sufficient conditions for \eqref{eq:sum_rule} can be expressed in terms of the domains of the functions (see, e.g.,~\cite[Corollary~16.50]{bauschke2017}). Specialising to the indicator functions to sets, that is, in the framework of Corollary~\ref{cor:zero of strengthenings normalcone}, the sum rule becomes
\begin{equation}\label{eq:strong_chip}
\sum_{i=1}^n N_{C_i}(x) = N_{\cap_{i=1}^nC_i}(x),\quad\forall x\in\Hilbert,
\end{equation}
which is known as the \emph{strong conical hull intersection (strong CHIP)} property. Specific sufficient conditions for \eqref{eq:strong_chip} can be found in \cite{burachik2005}.
\end{remark}

\section{Forward-backward-type methods}\label{s:fb methods}

In this section, we focus on the problem of computing
  \begin{equation}\label{eq:JABq}
   J_{\omega(A+B)}(q),
  \end{equation}
for some given $q\in\Hilbert$ and $\omega>0$, where $A\colon\Hilbert\setto\Hilbert$ is $\alpha_A$-maximally monotone and $B:\Hilbert\to\Hilbert$ is $\alpha_B$-monotone, single-valued and continuous.

In this situation, we can perform direct evaluations (forward steps) of $B$ and resolvent evaluations (backward steps) of $A$. For simplicity of exposition, the operator $B$ is assumed to have full domain which ensures maximality of $A+B$.

\begin{assumption}\label{a:1}
Let $\alpha_A,\alpha_B\in\mathbb{R}$ denote the monotonicity constants associated with the operators $A$ and $B$ in~\eqref{eq:JABq}, respectively. Suppose $\theta>0$ and $\sigma=(\sigma_A,\sigma_B)\in\mathbb{R}_{++}^2$ satisfy
$$\theta\alpha_A+\sigma_A>0\quad\text{and}\quad\theta\alpha_B+\sigma_B>0.$$
\end{assumption}

\begin{remark}\label{r:a1_v2}
For any $\alpha_A,\alpha_B\in\mathbb{R}$, there always exist $\theta,\sigma_A,\sigma_B\in\mathbb{R}_{++}$ satisfying Assumption~\ref{a:1}. Thus, Assumption~\ref{a:1} does not induce any restrictions on the operators $A$ and $B$ in \eqref{eq:JABq}, but it may restrict the values of $\omega$ for which the resolvent in \eqref{eq:JABq} can be computed if $\alpha_A$ or $\alpha_B$ is negative. When $A$ and $B$ are monotone (i.e., $\alpha=(\alpha_A,\alpha_B)\in\mathbb{R}^2_+$), Assumption~\ref{a:1} is trivially satisfied.

In some circumstances, it may  suffice to assume that $\theta>0$ and $(\sigma_A,\sigma_B)\in\R^2$ satisfy the weaker assumption
\begin{equation}\label{a1_rel}
\sigma_A+\sigma_B>0, \quad \theta\alpha_A+\sigma_A\geq 0\quad\text{and}\quad\theta\alpha_B+\sigma_B\geq0.
\end{equation}
In this case, Proposition~\ref{prop:strength_resolvent}\eqref{prop:strength_resolvent_I} yields maximal monotonicity of the strengthenings, whereas Proposition~\ref{prop:zero of strengthenings} is still applicable. Thus, \eqref{a1_rel} may be sufficient for the convergence of some algorithms (see, e.g., Remark~\ref{r:fbf}). However, we shall develop our analysis under Assumption~\ref{a:1} since it significantly improves our results and simplifies our presentation.
\end{remark}

In the following result we establish the maximality of the sum $A+B$ within the framework of this section. Thus, under Assumption~\ref{a:1}, condition $q\in\ran(\Id+\frac{\theta}{\sigma_A+\sigma_B}(A+B))$ in Proposition~\ref{prop:zero of strengthenings} holds for every point~$q\in\Hilbert$.

\begin{lemma}\label{l:maxAB}
Let $A:\Hilbert\setto\Hilbert$ be maximally $\alpha_A$-monotone and let $B:\Hilbert\to\Hilbert$ be $\alpha_B$-monotone and continuous. Then $A+B$ is maximally $(\alpha_A+\alpha_B)$-monotone. Consequently, if $\theta,\sigma_A,\sigma_B\in\R_{++}$ satisfy Assumption~\ref{a:1}, the resolvent $J_{\frac{\theta}{\sigma_A+\sigma_B}(A+B)}$ has full domain.
\end{lemma}
\begin{proof}
Consider the operators $\widetilde A:=A-\alpha_A\Id$ and $\widetilde B:=B-\alpha_B\Id$. Then $\widetilde A$ is maximally monotone and $\widetilde B:\Hilbert\to\Hilbert$ is monotone and continuous. Thus, by \cite[Corollary~20.28]{bauschke2017} we get that $\widetilde B$ is maximally monotone and $\dom \widetilde B=\Hilbert$. It then follows from \cite[Corollary~25.5]{bauschke2017} that
$$\widetilde A+\widetilde B= A+B-(\alpha_A+\alpha_B)\Id,$$
is maximally monotone. Hence, $A+B$ is maximally $(\alpha_A+\alpha_B)$-monotone. The remaining assertion follows from Proposition~\ref{prop:Jalpha}, since
$$1+\frac{\theta}{\sigma_A+\sigma_B}(\alpha_A+\alpha_B)>0,$$
by Assumption~\ref{a:1}.
\end{proof}

\hspace{-1.2mm}Recall that a sequence $(x_k)\subseteq\Hilbert$ converges \emph{linearly} to $x\in\Hilbert$ if $\limsup_{k\to\infty}\frac{\|x_{k+1}-x\|}{\|x_k-x\|}<1$.

\begin{theorem}[Forward-backward method]\label{th:fb}
Let $A\colon\Hilbert\setto\Hilbert$ be maximally $\alpha_A$-monotone and let $\gamma,\theta,\sigma_A,\sigma_B\in\R_{++}$. Suppose one of the following holds:
\begin{enumerate}[(i)]
\item\label{it:fb i} $B\colon\Hilbert\to\Hilbert$ is $\alpha_B$-monotone and $\kappa$-Lipschitz and $\gamma\in{\bigl]0,\frac{2(\theta\alpha_B+\sigma_B)}{(\theta\kappa+\sigma_B)^2}\bigr[}$.
\item\label{it:fb ii} $B\colon\Hilbert\to\Hilbert$ is $\beta$-cocoercive for $\beta>0$ and $\gamma\in{\bigl]0,\frac{2\beta}{\theta+\beta\sigma_B}\bigr[}$. In this case, set $\alpha_B:=0$.
\end{enumerate}
Suppose that Assumption~\ref{a:1} holds. For any initial point $x_0\in\Hilbert$, consider the sequence $(x_k)$ given by
\begin{equation}\label{eq:fb}
 x_{k+1} = J_{\frac{\gamma\theta}{1+\gamma\sigma_A}A}\left(\frac{1}{1+\gamma\sigma_A}\left((1-\gamma \sigma_B)x_k-\gamma\theta B(x_k)+\gamma(\sigma_A+\sigma_B)q\right) \right),\quad\forall k\in\mathbb{N}.
\end{equation}
Then $(x_k)$ converges linearly to $J_{\frac{\theta}{\sigma_A+\sigma_B}(A+B)}(q)$.
\end{theorem}
\begin{proof}
According to Proposition~\ref{prop:strength_resolvent}, $(A_{-q})^{(\theta,\sigma_A)}$ is maximally $(\theta\alpha_A+\sigma_A)$-strongly monotone and $ J_{\gamma A_{-q}^{(\theta,\sigma_A)}}$ is single-valued with full domain. For the operator $(B_{-q})^{(\theta,\sigma_B)}$, we distinguish two cases which correspond to \eqref{it:fb i} and \eqref{it:fb ii}, respectively.
\begin{enumerate}[(i)]
\item By Proposition~\ref{prop:strength_resolvent}\eqref{prop:strength_resolvent_I}, $(B_{-q})^{(\theta,\sigma_B)}$ is $(\theta\alpha_B+\sigma_B)$-strongly monotone. Since $B$ is $\kappa$-Lipschitz, Theorem~\ref{th:lip coco}\eqref{it:Lip of streng} implies that $(B_{-q})^{(\theta,\sigma_B)}$ is $(\kappa\theta+\sigma_B)$-Lipschitz. In this setting, we assume that $\gamma < \frac{2(\theta\alpha_B+\sigma_B)}{(\kappa\theta+\sigma_B)^2}$.

\item Since $B$ is $\beta$-cocoercive, it is monotone; i.e., $\alpha_B$-monotone for $\alpha_B=0$. Moreover, since $\sigma_B>0$, Theorem~\ref{th:lip coco}\eqref{it:coco of streng} shows that $(B_{-q})^{(\theta,\sigma_B)}$ is $\mu$-cocoercive for $\mu:=(\theta/\beta+\sigma_B)^{-1}$. In this setting, we assume that $\gamma<2\mu=\frac{2\beta}{\theta+\beta\sigma_B}$.
\end{enumerate}
Now, let $y_0:=\frac{1}{\theta}(x_0-q)$ and let $(y_k)$ be the sequence generated according to
\begin{equation}\label{eq:fb strengthening}
 y_{k+1} = J_{\gamma (A_{-q})^{(\theta,\sigma_A)}}\left(y_k-\gamma (B_{-q})^{(\theta,\sigma_B)}(y_k)\right),\quad\forall k\in\mathbb{N}.
\end{equation}
In both of the above settings, \cite[Proposition~26.16]{bauschke2017} shows that $(y_k)$ converges linearly to the unique element in $\zer\left((A_{-q})^{(\theta,\sigma_A)}+(B_{-q})^{(\theta,\sigma_B)}\right)$ which, by Proposition~\ref{prop:zero of strengthenings}, is equal to \break $\frac{1}{\theta}(J_{\frac{\theta}{\sigma_A+\sigma_B}(A+B)}(q)-q)$. By setting $x_k:=\theta y_k+q$ for all $k\in\mathbb{N}$ and using Proposition~\ref{prop:strength_resolvent}\eqref{prop:strength_resolvent_II}, one obtains \eqref{eq:fb} from \eqref{eq:fb strengthening}. It follows that $(x_k)$ is given by \eqref{eq:fb} and converges linearly to $J_{\frac{\theta}{\sigma_A+\sigma_B}(A+B)}(q)$.
\end{proof}

\begin{remark}
Consider the setting of Theorem~\ref{th:fb}\eqref{it:fb i}. $B^{(\theta,-\theta\alpha_B)}=B^{(\theta,\sigma_B)}-(\theta\alpha_B+\sigma_B)\Id$ is $\theta(\kappa+|\alpha_B|)$-Lipschitz continuous by Theorem~\ref{th:lip coco}\eqref{it:Lip of streng}. In this setting, Chen and Rockafellar \cite[Theorem~2.4]{chen1997convergence} proved that $(x_k)$ converges linearly when $\gamma>0$ satisfies
$$ \gamma^{-1} > \frac{\theta(\alpha_B-\alpha_A)+\sigma_B-\sigma_A}{2}+\frac{\theta(\kappa+|\alpha_B|)}{2}\max\left\{1,\frac{\theta(\kappa+|\alpha_B|)}{\theta(\alpha_B+\alpha_A)+\sigma_B+\sigma_A}\right\}.\qedhere$$
\end{remark}

\begin{theorem}[Forward-backward-forward method]\label{th:fbf}
Let $A\colon\Hilbert\setto\Hilbert$ be maximally $\alpha_A$-monotone, and let $B\colon\Hilbert\to\Hilbert$ be $\alpha_B$-monotone and $\kappa$-Lipschitz. Let $\theta,\sigma_A,\sigma_B\in\R_{++}$ be such that Assumption~\ref{a:1} holds and let $\gamma\in{]0,\frac{1}{\theta\kappa+\sigma_B}[}$. For any initial point $x_0\in\Hilbert$, consider the sequence $(x_k)$ given by
\begin{equation}\label{eq:fbf}
\left\{\begin{aligned}
     y_k &= J_{\frac{\gamma\theta}{1+\gamma\sigma_A}A}\left(\frac{1}{1+\gamma\sigma_A}\left((1-\gamma\sigma_B)x_k-\gamma\theta B(x_k)+\gamma(\sigma_A+\sigma_B)q\right)\right) \\
     x_{k+1} &= (1-\gamma\sigma_B)y_k +\gamma\sigma_Bx_k -\gamma\theta B(y_k) + \gamma \theta B(x_k).
\end{aligned}\right.
\end{equation}
Then $(x_k)$ converges linearly to $J_{\frac{\theta}{\sigma_A+\sigma_B}(A+B)}(q)$.
\end{theorem}
\begin{proof}
The proof is similar to Theorem~\ref{th:fb}\eqref{it:fb i} but uses \cite[Theorem~3.4(c)]{tseng2000modified} in place of \cite[Proposition~26.16]{bauschke2017}. Indeed, let $u_0:=\frac{1}{\theta}(x_0-q)$ and consider the sequence $(u_k)$ given by
\begin{equation}\label{eq:fbf_uk}
 \left\{\begin{aligned}
     v_{k} &= J_{\gamma (A_{-q})^{(\theta,\sigma_A)}}\left(u_k-\gamma (B_{-q})^{(\theta,\sigma_B)}(u_k)\right) \\
     u_{k+1} &= v_k -\gamma (B_{-q})^{(\theta,\sigma_B)}(v_k) + \gamma (B_{-q})^{(\theta,\sigma_B)}(u_k). \end{aligned}\right.
\end{equation}
Since $(B_{-q})^{(\theta,\sigma_B)}$ is $(\theta\kappa+\sigma_B)$-Lipschitz, applying \cite[Theorem~3.4(c)]{tseng2000modified} (or \cite[Theorem~26.17]{bauschke2017}) shows that the sequence $(u_k)$ converges linearly to the unique element of $\zer\left((A_{-q})^{(\theta,\sigma_A)}+(B_{-q})^{(\theta,\sigma_B)}\right)$ which is equal to $\frac{1}{\theta}(J_{\frac{\theta}{\sigma_A+\sigma_B}(A+B)}(q)-q)$.
By Proposition~\ref{prop:strength_resolvent}\eqref{prop:strength_resolvent_II}, we deduce that
  $$ \left\{\begin{aligned}
     \theta v_{k} +q&= J_{\frac{\gamma\theta}{1+\gamma\sigma_A}A}\left(\frac{1}{1+\gamma\sigma_A}\left((1-\gamma\sigma_B)(\theta u_k+q)-\gamma\theta B(\theta u_k+q)+\gamma(\sigma_A+\sigma_B)q\right)\right) \\
     \theta u_{k+1} +q&= (1-\gamma\sigma_B)(\theta v_k+q)+\gamma\sigma_B(\theta u_k+q) -\gamma\theta B(\theta v_k+q) + \gamma\theta B(\theta u_k+q). \end{aligned}\right. $$
By setting $(x_k,y_k):=(\theta u_k+q,\theta v_k+q)$, we obtain \eqref{eq:fbf} and deduce that $(x_k)$ converges linearly to $J_{\frac{\theta}{\sigma_A+\sigma_B}(A+B)}(q)$.
\end{proof}

\begin{remark}\label{r:fbf}
Weak converge of the algorithm in Theorem~\ref{th:fbf} can still be obtained if we assume the restrictions on the parameters in \eqref{a1_rel} instead of Assumption~\ref{a:1}. Indeed, as stated in Remark~\ref{r:a1_v2},  $(A_{-q})^{(\theta,\sigma_A)}$ is maximally monotone, while $(B_{-q})^{(\theta,\sigma_B)}$ remains $(\theta\kappa+\sigma_B)$-Lipschitz. Therefore, by applying \cite[Theorem~26.17(ii)]{bauschke2017}, we obtain that the sequence $(u_k)$ generated by \eqref{eq:fbf_uk} converges weakly to a zero of $(A_{-q})^{(\theta,\sigma_A)}+(B_{-q})^{(\theta,\sigma_B)}$, provided that this set of zeros is nonempty, or equivalently, if $q\in\ran(\Id+\frac{\theta}{\sigma_A+\sigma_B}(A+B))$. Observe that the latter automatically holds under Assumption~\ref{a:1}, in view of Lemma~\ref{l:maxAB}.
\end{remark}

The following result concerns a strengthened version of the adaptive \emph{\textbf{G}olden \textbf{Ra}tio \textbf{Al}gorithm (GRAAL)} for variational inequalities.
\begin{theorem}[Adaptive GRAAL]
Suppose $\dim\Hilbert<+\infty$. Let $g\colon\Hilbert\to {]-\infty,+\infty]}$ be proper, lsc and $\alpha_A$-convex, and let $B\colon\Hilbert\to\Hilbert$ be $\alpha_B$-monotone and locally Lipschitz continuous. Let $\theta,\sigma_A,\sigma_B\in\R_{++}$ be such that Assumption~\ref{a:1} holds. Choose $x_0,x_1\in\Hilbert$, $\gamma_0,\overline{\gamma}\in\R_{++}$, $\phi\in\bigl]1,\frac{1+\sqrt{5}}{2}\bigr]$. Set $\bar{x}_0=x_1$, $\gamma_{-1}=\phi\gamma_0$ and $\rho=\frac{1}{\phi}+\frac{1}{\phi^2}$. For $k\geq 1$, consider
\begin{equation}\label{eq:graal1}
\left\{\begin{aligned}
\gamma_k &= \min\left\{\rho\gamma_{k-1},\frac{\phi^2}{4\gamma_{k-2}}\frac{\|x_k-x_{k-1}\|^2}{\|\theta\bigl(B(x_k)-B(x_{k-1})\bigr)+\sigma_B(x_k-x_{k-1})\|^2},\overline{\gamma}\right\} \\
\bar{x}_k &= \frac{(\phi-1)x_k+\bar{x}_{k-1}}{\phi} \\
x_{k+1}   &=\prox_{\frac{\gamma_k\theta}{1+\gamma_k\sigma_A}g}\left(\frac{1}{1+\gamma_k\sigma_A}\left(\bar{x}_k-\gamma_k\sigma_B x_k-\gamma_k \theta B(x_k)+\gamma_k(\sigma_A+\sigma_B)q\right)\right).
\end{aligned}\right.
\end{equation}
Then $(x_k)$ and $(\bar{x}_k)$ converge to $J_{\frac{\theta}{\sigma_A+\sigma_B}(\partial g+B)}(q)$.
\end{theorem}
\begin{proof}
The proof is similar to Theorem~\ref{th:fb}\eqref{it:fb i} with \cite[Theorem~2]{malitsky2019golden} applied. To this end, set $\bar{g}(z):=\frac{1}{\theta}g(\theta z+q)+\frac{\sigma_A}{2}\|z\|^2$ and $F:=(B_{-q})^{(\theta,\sigma_B)}$. Since $g$ is proper, lsc and $\alpha_A$-convex, $\partial\bar{g}=((\partial g)_{-q})^{(\theta,\sigma_A)}$ is maximally $(\theta\alpha_A+\sigma_A)$-monotone by Proposition~\ref{prop:strength_resolvent}\eqref{prop:strength_resolvent_I} (see also Example~\ref{ex:proxproj}\eqref{ex:prox} and Remark~\ref{r:gbar}). Note also that since $B$ is locally Lipschitz, so is $F$ by Theorem~\ref{th:lip coco}\eqref{it:local Lip of streng}. Now, consider the variational inequality
 $$ \text{find~~}z^*\in\Hilbert\text{~~such that~~}\langle F(z),z-z^*\rangle+\bar{g}(z)-\bar{g}(z^*)\geq 0\quad\forall z\in\Hilbert, $$
which is equivalent to finding a point $z^*\in\zer(\partial\bar{g}+F)$ (which is nonempty by Proposition~\ref{prop:zero of strengthenings} and Lemma~\ref{l:maxAB}). Set $(z_0,z_1):=\frac{1}{\theta}(x_0-q,x_1-q)$ and $\bar{z}_0=z_1$. Consider the sequences generated by
\begin{equation}\label{eq:graal2}
\left\{\begin{aligned}
\gamma_k &= \min\left\{\rho\gamma_{k-1},\frac{\phi^2}{4\gamma_{k-2}}\frac{\|z_k-z_{k-1}\|^2}{\|F(z_k)-F(z_{k-1})\|^2},\overline{\gamma}\right\} \\
\bar{z}_k &= \frac{(\phi-1)z_k+\bar{z}_{k-1}}{\phi} \\
z_{k+1}   &= \prox_{\gamma_k\bar{g}}\bigl( \bar{z}_k-\gamma_kF(z_k) \bigr).
\end{aligned}\right.
\end{equation}
Then, according to \cite[Theorem~2]{malitsky2019golden}, $(z_k)$ and $(\bar{z}_k)$ converge to a point $z^*\in \zer(\partial\bar{g}+F)$. Since $\prox_{\gamma_k\bar{g}}=J_{\gamma_k\partial\bar{g}}$ by Example~\ref{ex:proxproj}\eqref{ex:prox}, Proposition~\ref{prop:strength_resolvent}\eqref{prop:strength_resolvent_II} gives
\begin{align*}
 z_{k+1} &=\frac{1}{\theta}\left(\prox_{\frac{\gamma_k\theta}{1+\gamma_k\sigma_A}g}\left(\frac{\theta}{1+\gamma_k\sigma_A}\left(\bar{z}_k-\gamma_k B(\theta z_k+q)-\gamma_k\sigma_Bz_k\right)+q\right)-q\right).
\end{align*}
Setting $(x_k,\bar{x}_k):=(\theta z_k+q,\theta \bar{z}_k+q)$ in \eqref{eq:graal2} gives \eqref{eq:graal1}. It follows that $(x_k)$ and $(\bar{x}_k)$ converge to $x^*:=\theta z^*+q$ and $x^*=J_{\frac{\theta}{\sigma_A+\sigma_B}(\partial g+B)}(q)$, by Proposition~\ref{prop:zero of strengthenings}.
\end{proof}

\begin{theorem}[Primal-dual method]\label{th:primaldual}
Suppose $\dim\Hilbert<+\infty$ and $\dim\Hilbert'<+\infty$. Let $\sigma>0$, $g\colon\Hilbert\to{]-\infty,+\infty]}$ be proper, lsc and $(-\sigma)$-convex, $\phi\colon\Hilbert'\to {]-\infty,+\infty]}$ be proper, lsc and convex, and $K\colon\Hilbert\to\Hilbert'$ be a linear operator. Suppose $q\in\ran\left(\Id+\frac{1}{\sigma}(\partial g+K^*\circ\partial\phi\circ K)\right)$.
Choose $(x_0,y_0)\in\Hilbert\times\Hilbert', \lambda\in[0,1]$ and $\gamma,\tau>0$ such that $\gamma\tau\|K\|^2<1$. Set $\bar{x}_0:=x_0$ and, for $k\geq 1$, consider
\begin{equation}\label{eq:s pd}
\left\{\begin{aligned}
      y_{k+1} &= \prox_{\gamma\phi^*}\bigl(y_k+\gamma K\bar{x}_k\bigr) \\
      x_{k+1} &= \prox_{\frac{\tau}{1+\tau\sigma}g}\left(\frac{1}{1+\tau\sigma}(x_k-\tau K^*y_{k+1}+\tau\sigma q)\right) \\
\bar{x}_{k+1} &= x_{k+1}+\lambda\bigl(x_{k+1}-x_k\bigr). \\
\end{aligned}\right.
\end{equation}
Then $(x_k)$ and $(\bar{x}_k)$ converge to $\prox_{\frac{1}{\sigma}(g+\phi\circ K)}(q)$.
\end{theorem}

\begin{proof}
Denote $\bar{g}(x):=g(x)+\frac{\sigma}{2}\|x-q\|^2$. Then $\bar{g}$ is convex and
$$ \partial\bar{g} = \partial g+\sigma(\Id-q) = (\partial g_{-q})^{(1,\sigma)}\circ (\Id-q). $$
 Since $q\in\ran\left(\Id+\frac{1}{\sigma}(\partial g+K^*\circ\partial\phi \circ K)\right)$, there exist points $(x,y)\in\Hilbert\times\Hilbert'$ such that
\begin{equation}\label{eq:pd saddle point}
 \left\{\begin{array}{rl}
   0 &\in \partial\bar{g}(x)+K^*y \\
   y &\in \partial\phi(Kx) \\
 \end{array}\right.
 \iff
 \left\{\begin{array}{rl}
   -K^*y &\in \partial\bar{g}(x) \\
   Kx &\in (\partial\phi)^{-1}(y)=\partial\phi^*(y). \\
 \end{array}\right.
\end{equation}
Using \cite[Theorem~1]{chambolle2011first}, we deduce that the sequences $(x_k)$ and $(y_k)$ given by
\begin{equation}\label{eq:pd gbar}
\left\{\begin{aligned}
      y_{k+1} &= \prox_{\gamma\phi^*}\bigl(y_k+\gamma K\bar{x}_k\bigr) \\
      x_{k+1} &= \prox_{\tau\bar{g}}\bigl(x_k-\tau K^*y_{k+1}\bigr) \\
\bar{x}_{k+1} &= x_{k+1}+\lambda\bigl(x_{k+1}-x_k\bigr) \\
\end{aligned}\right.
\end{equation}
converge to points $x\in\Hilbert$ and $y\in\Hilbert'$, respectively, which satisfy \eqref{eq:pd saddle point}. Moreover, we also have $x=J_{\frac{1}{\sigma}(\partial g+K^*\circ\partial\phi\circ K)}(q)=\prox_{\frac{1}{\sigma}(g+\phi\circ K)}(q).$
By applying \cite[Proposition~23.17(iii)]{bauschke2017} followed by Proposition~\ref{prop:strength_resolvent}\eqref{prop:strength_resolvent_II}, we obtain
\begin{align*}
\prox_{\tau\bar{g}}=J_{\tau(\partial g_{-q})^{(1,\sigma)}\circ(\Id-q)}
&= q+J_{\tau (\partial g_{-q})^{(1,\sigma)}}\circ(\Id-q) \\
&= q+J_{\frac{\tau}{1+\tau\sigma}\partial g}\circ\left(\frac{1}{1+\tau\sigma}\Id+q\right)\circ(\Id-q)-q \\
&=\prox_{\frac{\tau}{1+\tau\sigma}g}\circ\left(\frac{1}{1+\tau\sigma}\left(\Id+\tau\sigma q\right)\right).
\end{align*}
Substituting this expression into \eqref{eq:pd gbar} gives \eqref{eq:s pd}, and the claimed result follows.
\end{proof}

\begin{remark}\label{r:primaldual}
The assumption $q\in\ran\left(\Id+\frac{1}{\sigma}(\partial g+K^*\circ\partial\phi\circ K)\right)$ in Theorem~\ref{th:primaldual} holds under standard constraint qualifications (e.g., $K\dom g\cap\operatorname{cont}\phi\neq\emptyset$, where $\operatorname{cont}\phi$ denotes set of points where $\phi$ is continuous). Indeed, we have
$x=\prox_{\frac{1}{\sigma}(g+\phi\circ K)}(q) \allowbreak \iff 0\in x-q+\frac{1}{\sigma}\partial(g+\phi\circ K)(x). $
When the subdifferential sum-rule holds, the latter is equivalent to
$ 0\in x-q+\frac{1}{\sigma}\bigl(\partial g(x)+K^*\circ\partial\phi(Kx)\bigr)$, which implies $q\in\ran\left(\Id+\frac{1}{\sigma}(\partial g+K^*\circ\partial\phi\circ K)\right)$.
\end{remark}

\begin{remark}
Using the framework devised in this section, several other algorithms of forward-backward-type for finding the resolvent of the sum of two monotone operators can be derived. Examples of addition algorithms for general monotone operators include those based on shadow Douglas--Rachford splitting \cite{csetnek2019shadow}, forward reflected backward splitting \cite{malitsky2018forward}, reflected forward-backward splitting \cite{cevher2019reflected}, projective splitting with forward steps \cite{johnstone2019projective}, three-operator splitting \cite{davis2017three}, the forward-Douglas--Rachford method \cite{ryu2020finding}, and the backward-forward-reflected-backward method \cite{rieger2020backward}. Furthermore, in the special case where the operator $A$ is the normal cone to a convex set, algorithms based on Korpelevich's and on Popov's extragradient methods \cite{korpelevich1976extragradient,popov1980modification}, respectively, can also be derived.
\end{remark}

\section{Resolvent Iterations}\label{s:resolvent iterations}
In this section, we focus on the problem of computing
\begin{equation}\label{eq:JABq'}
  J_{\omega(A+B)}(q),
\end{equation}
for some given $q\in\Hilbert$ and $\omega>0$, where $A\colon\Hilbert\setto\Hilbert$ is maximally $\alpha_A$-monotone and $B\colon\Hilbert\setto\Hilbert$ is maximally $\alpha_B$-monotone. In this situation, we will assume that we have access to the resolvents of both $A$ and $B$. We will also consider the extension to three operators, that is, the problem of computing
\begin{equation}\label{eq:JABCq}
 J_{\omega(A+B+C)}(q),
\end{equation}
where, in addition, $C\colon\Hilbert\setto\Hilbert$ is maximally $\alpha_C$-monotone.

The following assumption is a variant of Assumption~\ref{a:1} with operator $B$ (potentially) set-valued, rather than single-valued.
\begin{assumption}\label{a:1'}
Let $\alpha_A,\alpha_B\in\mathbb{R}$ denote the monotonicity constants associated with the operators $A$ and $B$ in~\eqref{eq:JABq'}, respectively. Suppose $\theta>0$ and $\sigma=(\sigma_A,\sigma_B)\in\mathbb{R}_{++}^2$ satisfy
$$\theta\alpha_A+\sigma_A>0\quad\text{and}\quad\theta\alpha_B+\sigma_B>0.$$
\end{assumption}

\begin{theorem}[Douglas/Peaceman--Rachford algorithm]\label{th:DR}
Let $A,B\colon\Hilbert\setto\Hilbert$ be maximally $\alpha_A$-monotone and maximally $\alpha_B$-monotone, respectively. Let $\theta,\sigma_A,\sigma_B\in\mathbb{R}_{++}$ be such that Assumption~\ref{a:1'} holds, let $\gamma>0$, let $\lambda\in{]0,2]}$ and let $q\in\ran\left(\Id+\frac{\theta}{\sigma_A+\sigma_B}(A+B)\right)$. Given any arbitrary $x_0\in\Hilbert$, consider the sequences generated by
\begin{equation}\label{eq:alg_DR}
\left\{\begin{aligned}
    u_k &= J_{\frac{\gamma\theta}{1+\gamma\sigma_A}A}\left(\frac{1}{1+\gamma\sigma_A}(x_k+\gamma\sigma_Aq)\right) \\
    v_k &= J_{\frac{\gamma\theta}{1+\gamma\sigma_B}B}\left(\frac{1}{1+\gamma\sigma_B}(2u_k- x_k+\gamma\sigma_Bq)\right) \\
    x_{k+1} &= x_k + \lambda(v_k-u_k). \\
   \end{aligned}\right.
\end{equation}
Then $u_k\to J_{\frac{\theta}{\sigma_A+\sigma_B}(A+B)}(q)$ and $x_k\wto x$ with
$$J_{\frac{\gamma\theta}{1+\gamma\sigma_A}A}\left(\frac{1}{1+\gamma\sigma_A}(x+\gamma\sigma_Aq)\right)=J_{\frac{\theta}{\sigma_A+\sigma_B}(A+B)}(q).$$
\end{theorem}
\begin{proof}
Set $\bar{x}_0:=\frac{1}{\theta}(x_0-q)$ and consider the sequences
\begin{equation}\label{eq:alg_DR_modif}
\left\{\begin{aligned}
    \bar{u}_k &= J_{\gamma (A_{-q})^{(\theta,\sigma_A)}}(\bar{x}_k) \\
    \bar{v}_k &= J_{\gamma (B_{-q})^{(\theta,\sigma_B)}}(2\bar{u}_k-\bar{x}_k) \\
    \bar{x}_{k+1} &= \bar{x}_k + \lambda(\bar{v}_k-\bar{u}_k). \\
   \end{aligned}\right.
\end{equation}
We distinguish two cases based on the value of $\lambda$. First, suppose that $\lambda\in{]0,2[}$. By combining \cite[Theorem 26.11]{bauschke2017} with Proposition~\ref{prop:zero of strengthenings}, we get that $\bar{x}_k\wto \bar{x}$ and $\bar{u}_k\to \bar{u}$ with $$\bar{u}=J_{\gamma (A_{-q})^{(\theta,\sigma_A)}}(\bar{x})\in\zer\left((A_{-q})^{(\theta,\sigma_A)}+(B_{-q})^{(\theta,\sigma_B)}\right)=\left\{\frac{1}{\theta}\left(J_{\frac{\theta}{\sigma_A+\sigma_B}(A+B)}(q)-q\right)\right\}.$$
Here the strong convergence of $(\bar u_k)$ comes from \cite[Theorem 26.11(vi)]{bauschke2017} and the fact that both $(A_{-q})^{(\theta,\sigma_A)}$ and $(B_{-q})^{(\theta,\sigma_B)}$ are strongly monotone by Proposition~\ref{prop:strength_resolvent}\eqref{prop:strength_resolvent_I}.
Now, using Proposition~\ref{prop:strength_resolvent}\eqref{prop:strength_resolvent_II} and making the change of variables $(x_k,u_k,v_k):=(\theta\bar{x}_k+q,\theta\bar{u}_k+q,\theta\bar{v}_k+q)$, $u:=\theta{\bar{u}}+q$ and $x:=\theta\bar{x}+q$ the iteration in \eqref{eq:alg_DR_modif} reduces to \eqref{eq:alg_DR} and the result follows.

Next, consider the limiting case where $\lambda=2$. Observe that \eqref{eq:alg_DR_modif} can be expressed as
\begin{equation}\label{eq:PR}
\bar{x}_{k+1} = \left(2J_{\gamma (B_{-q})^{(\theta,\sigma_B)}}-\Id\right)\circ \left(2J_{\gamma (A_{-q})^{(\theta,\sigma_A)}} -\Id\right)(\bar{x}_k).
\end{equation}
Since $(A_{-q})^{(\theta,\sigma_A)}$ and $(B_{-q})^{(\theta,\sigma_B)}$ are maximally strongly monotone, their reflected resolvents, $$2J_{\gamma (A_{-q})^{(\theta,\sigma_A)}}-\Id \quad\text{and}\quad2J_{\gamma (B_{-q})^{(\theta,\sigma_B)}}-\Id,$$ are negatively averaged by \cite[Proposition~5.4]{giselsson_tight}. Their composition is therefore averaged by \cite[Proposition~3.12 and Remark~3.13]{giselsson_tight}. Thus, according to \cite[Proposition~5.16]{bauschke2017}, the sequence $(\bar{x}_k)$ generated by \eqref{eq:PR} weakly converges to a fixed point of the composition of the reflected resolvents. The strong convergence of the shadow sequence $(\bar{u}_k)$ is a consequence of~\cite[Proposition 26.13]{bauschke2017}. Finally, by making a change of variables as in the case $\lambda\in{]0,2[}$, the result follows.
\end{proof}

\begin{remark}\label{rem:AAMR}
Theorem~\ref{th:DR} was established in \cite[Theorem~3.2]{dao2019resolvent}, with the exception of the weak convergence of $(x_k)$ in the case $\lambda=2$. As we now explain, it covers the following two schemes from the literature as special cases. In both settings, we assume that $A$ and $B$ are maximally monotone operators and that $\alpha_A=\alpha_B=0$.
\begin{enumerate}[(i)]
\item Set $\theta:=\frac{1}{\beta}$ and $\sigma_A=\sigma_B:=\frac{1-\beta}{\gamma\beta}$
for any $\beta\in{]0,1[}$. Then \eqref{eq:alg_DR} becomes
\begin{equation*}
\left\{\begin{aligned}
    u_k &= J_{\gamma A}\left(\beta x_k + (1-\beta)q\right) \\
    v_k &= J_{\gamma B}\left(\beta(2u_k-x_k)+(1-\beta)q\right) \\
    x_{k+1} &= \left(1-\frac{\lambda}{2}\right)x_k + \frac{\lambda}{2}(2v_k-2u_k+x_k). \\
   \end{aligned}\right.
\end{equation*}
By taking $y_k:=\beta (x_k-q)$ and $\kappa:=\frac{\lambda}{2}\in{]0,1[}$, this can be rewritten as
\begin{equation*}
y_{k+1}=(1-\kappa)y_k+\kappa(2\beta J_{(\gamma B_{-q})}-\Id)(2\beta J_{(\gamma A_{-q})}-\Id)(y_k),
\end{equation*}
which coincides with the \emph{Averaged Alternating Modified Reflections (AAMR)} algorithm developed in~\cite{aragon2019computing}.

\item  Let $\lambda=2$, $\sigma_A=\sigma_B$, $\gamma=\frac{1}{\sigma_A}$, and $\theta=2\sigma_A$. By denoting $z_k:=\frac{1}{2}(x_k+q)$, \eqref{eq:alg_DR} can be written as
\begin{align*}
x_{k+1}&=x_k-2J_A(z_k)+2J_B\left(J_A(z_k)-\frac{1}{2}x_k+\frac{1}{2}q\right)\\
&=2z_k-q-2J_A(z_k)+2J_B\left(J_A(z_k)-z_k+q\right),
\end{align*}
which is equivalent to
$$z_{k+1}=z_k-J_A(z_k)+J_B\left(J_A(z_k)-z_k+q\right).$$
This coincides with the variant of the Douglas--Rachford algorithm proposed in~\cite{adly2019decomposition}, where it is referred to as ``\emph{Algorithm~$(\mathcal{A})$}''.
While Adly and Bourdin proved weak convergence of $(z_k)$ to a point $z\in (\Id+B\circ J_A)^{-1}(q)$, they only established weak convergence of $(J_A(z_k))$ to $J_A(z)\in J_{A+B}(q)$ (see \cite[Theorem~3]{adly2019decomposition}). In addition to weak convergence of $(z_k)$, Theorem~\ref{th:DR} shows that $(J_A(z_k))$ is actually strongly convergent to $J_{A+B}(q)$.\qedhere
\end{enumerate}
\end{remark}

We now derive a splitting method for computing the resolvent of the sum of three operators based on the scheme proposed by Ryu in~\cite[Section~4]{ryu2019uniqueness} for finding a zero of the sum. We provide a direct proof of its convergence in the infinite dimensional setting in Appendix~\ref{s:appendix}, which also provides new conditions ensuring the convergence of the limiting case $\lambda=1$.

The following assumption is the three operator analogue of Assumption~\ref{a:1'}.
\begin{assumption}\label{a:2}
Let $\alpha_A,\alpha_B,\alpha_C\in\mathbb{R}$ denote the monotonicity constants associated with the operators $A$, $B$ and $C$ in~\eqref{eq:JABCq}, respectively. Suppose $\theta>0$ and $\sigma=(\sigma_A,\sigma_B,\sigma_C)\in\mathbb{R}_{+}^3$ satisfy
$$\theta\alpha_A+\sigma_A>0,\quad \theta\alpha_B+\sigma_B>0,\quad \theta\alpha_C+\sigma_C>0.$$
\end{assumption}
\begin{remark}
As was the case in Remark~\ref{r:a1_v2}, for any $\alpha_A,\alpha_B,\alpha_C\in\mathbb{R}$, there always exist $\theta,\sigma_A,\sigma_B,\sigma_C\in\mathbb{R}_{++}$ satisfying Assumption~\ref{a:2}.
Thus, Assumption~\ref{a:2} does not induce any restrictions on the operators $A,B$ and $C$ in \eqref{eq:JABCq}, but it may restrict the values of $\omega$ for which the resolvent in \eqref{eq:JABCq} can be computed if $\alpha_A$, $\alpha_B$ or $\alpha_C$ is negative. When $A, B$ and $C$ are monotone (i.e., when $\alpha=(\alpha_A,\alpha_B,\alpha_C)\in\mathbb{R}^3_{+})$, Assumption~\ref{a:2} trivially holds.
\end{remark}

\begin{theorem}[Ryu splitting]\label{th:ryu}
Let $A,B,C\colon\Hilbert\setto\Hilbert$ be maximally $\alpha_A$-monotone, maximally $\alpha_B$-monotone, and maximally $\alpha_C$-monotone, respectively. Let $\theta,\sigma_A,\sigma_B,\sigma_C\in\mathbb{R}_{++}$ be such that Assumption~\ref{a:2} holds, let $\gamma>0$, and let $\lambda\in{]0,1]}$. Suppose $q\in\ran\left(\Id+\frac{\gamma\theta}{\sigma_A+\sigma_B+\sigma_C}(A+B+C)\right)$. Given any $x_0,y_0\in \Hilbert$, consider the sequences
\begin{equation}\left\{\begin{aligned}
    u_k &= J_{\frac{\gamma\theta}{1+\sigma_A}A}\left(\frac{1}{1+\gamma\sigma_A} x_k+\frac{\gamma\sigma_A}{1+\gamma\sigma_A}q \right) \\
    v_k &= J_{\frac{\gamma\theta}{1+\gamma\sigma_B}B}\left(\frac{1}{1+\gamma\sigma_B}(u_k+y_k)-\frac{1-\gamma\sigma_B}{1+\gamma\sigma_B}q\right) \\
    w_k &= J_{\frac{\gamma\theta}{1+\gamma\sigma_C}C}\left(\frac{1}{1+\gamma\sigma_C}(u_k-x_k+v_k-y_k)+q \right) \\
    x_{k+1} &= x_k + \lambda(w_k-u_k) \\
    y_{k+1} &= y_k + \lambda(w_k-v_k).
   \end{aligned}\right.\label{eq:ryu sequences}
\end{equation}
Then $u_k\to J_{\frac{\theta}{\sigma_A+\sigma_B+\sigma_C}(A+B+C)}(q)$ as $k\to+\infty$. Further, if $\lambda\neq 1$, then $x_k\wto x$ with
$$ J_{\frac{\gamma}{1+\gamma\sigma_A}A} \left(\frac{1}{1+\gamma\sigma_A}x+\frac{\gamma\sigma_A}{1+\gamma\sigma_A}q\right)=J_{\frac{\theta}{\sigma_A+\sigma_B+\sigma_C}(A+B+C)}(q). $$
\end{theorem}
\begin{proof}
Set $(\bar{x}_0,\bar{y}_0):=\frac{1}{\theta}(x_0-q,y_0-q)$ and consider the sequences
\begin{equation}\label{eq:ryu strengthenings}
 \left\{\begin{aligned}
    \bar{u}_k &= J_{\gamma (A_{-q})^{(\theta,\sigma_A)}}(\bar{x}_k) \\
    \bar{v}_k &= J_{\gamma (B_{-q})^{(\theta,\sigma_B)}}(\bar{u}_k+\bar{y}_k) \\
    \bar{w}_k &= J_{\gamma (C_{-q})^{(\theta,\sigma_C)}}(\bar{u}_k-\bar{x}_k+\bar{v}_k-\bar{y}_k) \\
    \bar{x}_{k+1} &= \bar{x}_k + \lambda(\bar{w}_k-\bar{u}_k) \\
    \bar{y}_{k+1} &= \bar{y}_k + \lambda(\bar{w}_k-\bar{v}_k).
   \end{aligned}\right.
\end{equation}
By Assumption~\ref{a:2}, the sum $(A_{-q})^{(\theta,\sigma_A)}+(B_{-q})^{(\theta,\sigma_B)}+ (C_{-q})^{(\theta,\sigma_C)}$ is $\alpha$-strongly monotone for $\alpha := \bigl(\theta\alpha_A+\sigma_A\bigr)+\bigl(\theta\alpha_B+\sigma_B\bigr)+\bigl(\theta\alpha_C+\sigma_C\bigr)>0$. By assumption  $q\in\ran\left(\Id+\frac{\gamma\theta}{\sigma_A+\sigma_B+\sigma_C}(A+B+C)\right)$ and hence Proposition~\ref{prop:zero of strengthenings} implies
 $$\zer\left( (A_{-q})^{(\theta,\sigma_A)}+(B_{-q})^{(\theta,\sigma_B)}+(C_{-q})^{(\theta,\sigma_C)}\right)=\left\{\frac{1}{\theta}\left(J_{\frac{\theta}{\sigma_A+\sigma_B+\sigma_C}(A+B+C)}(q)-q\right)\right\}.$$
Appealing to Theorem~\ref{th:ryu appendix}\eqref{it:ryu a ii} gives $u_k\to\bar{u}$, where
$$
\bar{u}\in\zer\left( ( A_{-q})^{(\theta,\sigma_A)}+(B_{-q})^{(\theta,\sigma_B)}+(C_{-q})^{(\theta,\sigma_C)}\right). $$
Further, if $\lambda\in{]0,1[}$, then Theorem~\ref{th:ryu appendix}\eqref{it:ryu a i}
gives $(\bar{x}_k,\bar{y}_k)\wto (\bar{x},\bar{y})$ where $\bar{x}$ satisfies
$$
\bar{u}:=J_{\gamma (A_{-q})^{(\theta,\sigma_A)}}(\bar{x})\in\zer\left( ( A_{-q})^{(\theta,\sigma_A)}+(B_{-q})^{(\theta,\sigma_B)}+(C_{-q})^{(\theta,\sigma_C)}\right). $$
Finally, by applying Proposition~\ref{prop:strength_resolvent}\eqref{prop:strength_resolvent_II}, we may write \eqref{eq:ryu strengthenings} as
$$ \left\{\begin{aligned}
    \theta \bar{u}_k &= J_{\frac{\gamma\theta}{1+\gamma\sigma_A}A}\left(\frac{\theta}{1+\gamma\sigma_A} \bar{x}_k + q\right) -q \\
    \theta \bar{v}_k &= J_{\frac{\gamma\theta}{1+\gamma\sigma_B}B}\left(\frac{\theta}{1+\gamma\sigma_B}(\bar{u}_k+\bar{y}_k)+q\right)-q \\
    \theta \bar{w}_k &= J_{\frac{\gamma\theta}{1+\gamma\sigma_C}C}\left(\frac{\theta}{1+\gamma\sigma_C}(\bar{u}_k-\bar{x}_k+\bar{v}_k-\bar{y}_k)+q\right)-q
   \end{aligned}\right. $$
and $\theta J_{\gamma(A_{-q})^{(\theta,\sigma_A)}}(\bar{x})=J_{\frac{\gamma}{1+\gamma\sigma_A}A} \left(\frac{\theta}{1+\gamma\sigma_A}\bar{x}+q\right)-q$. The claimed result follows by making the change of variables $(x_k,y_k,u_k,v_k,w_k):=(\theta\bar{x}_k+q,\theta\bar{y}_k+q,\theta\bar{u}_k+q,\theta\bar{v}_k+q,\theta\bar{w}_k+q)$ for all $k\in\mathbb{N}$ and $(x,u):=(\theta\bar{x}+q,\theta\bar{u}+q)$.
\end{proof}

\begin{remark}
Consider the case with $B=0$ and $\sigma_B=0$. Although this setting is not covered by Assumption~\ref{a:2} as $\sigma_B\not>0$, it is easily covered by an extension analogous to the one described in Remark~\ref{r:a1_v2}. In such a case, the sequence $(v_k)$ in \eqref{eq:ryu sequences} simplifies to
$$ v_k = u_k+y_k-q.$$
Using this identity to eliminate $(v_k)$ and $(y_k)$ from \eqref{eq:ryu sequences} gives
\begin{equation*}\left\{\begin{aligned}
    u_k &= J_{\frac{\gamma\theta}{1+\sigma_A}A}\left(\frac{1}{1+\gamma\sigma_A} x_k+\frac{\gamma\sigma_A}{1+\gamma\sigma_A}q \right) \\
    w_k &= J_{\frac{\gamma\theta}{1+\gamma\sigma_C}C}\left(\frac{1}{1+\gamma\sigma_C}(2u_k-x_k)+\frac{\gamma\sigma_C}{1+\gamma\sigma_C}q \right) \\
    x_{k+1} &= x_k + \lambda(w_k-u_k),
   \end{aligned}\right.
\end{equation*}
which coincides with \eqref{eq:alg_DR} applied to the operators $A$ and $C$.
\end{remark}

\section{Applications}\label{s:applications}
In this section, we consider three applications of the framework developed in the previous sections. All computations were run on a machine running Ubuntu 18.04.5 LTS
with an Intel Core i7-8665U CPU and 16GB of memory.

\subsection{Best Approximation Problems}
Let $C_1,\dots,C_m\subseteq\Hilbert$ be closed and convex sets with $\cap_{i=1}^mC_i\neq\emptyset$. Given a point $q\in\Hilbert$, the \emph{best approximation problem} is
\begin{equation}\label{eq:bap}
 \min_{x\in\Hilbert}\frac{1}{2}\|x-q\|^2\text{~such that~}x\in \bigcap_{i=1}^mC_i,
\end{equation}
which is equivalent to the formally unconstrained problem
 $$ \min_{x\in\Hilbert}\frac{1}{2}\|x-q\|^2+\sum_{i=1}^m\iota_{C_i}(x). $$
We therefore see that solving \eqref{eq:bap} is equivalent to computing the proximity operator of $\sum_{i=1}^m\iota_{C_i}$. Thus, assuming the strong CHIP holds (see Remark~\ref{r:sCHIP}), the solution $x$ of \eqref{eq:bap} is given by
 $$ x = J_{\sum_{i=1}^mN_{C_i}}(q)=\prox_{\sum_{i=1}^m\iota_{C_i}}(q). $$
Using results from Section~\ref{s:resolvent iterations}, we obtain the following projection algorithm for best approximation which appears to be new. It will be referred to as \emph{S-Ryu} (which stands for \emph{strengthened-Ryu}).
\begin{corollary}[Best approximation with three sets -- S-Ryu]\label{cor:ryu_3sets}
Let $C_1,C_2,C_3\subseteq\Hilbert$ be closed and convex sets with nonempty intersection, and suppose that $q\in\ran(I+N_{C_1}+N_{C_2}+N_{C_3})$. Let $\beta\in{]0,1[}$ and $\lambda\in{]0,1]}$. Given any $x_0,y_0\in\Hilbert$, consider the sequences
\begin{equation} \left\{\begin{aligned}
    u_k &= P_{C_1}\left(\beta x_k+(1-\beta)q \right) \\
    v_k &= P_{C_2}\left(\beta(u_k+y_k)-(2\beta-1)q\right) \\
    w_k &= P_{C_3}\left(\beta(u_k-x_k+v_k-y_k)+q \right) \\
    x_{k+1} &= x_k + \lambda(w_k-u_k) \\
    y_{k+1} &= y_k + \lambda(w_k-v_k).
   \end{aligned}\right.\label{eq:Ryu_best_approx}
\end{equation}
Then $u_k\to u=P_{C_1\cap C_2\cap C_3}(q)$ as $k\to +\infty$. Furthermore, if $\lambda\neq 1$, then $x_k\wto x$ with $P_{C_1}\left(\beta x+(1-\beta)q\right)=u. $
\end{corollary}
\begin{proof}
Set $A=N_{C_1}, B=N_{C_2}, C=N_{C_3}$, $\sigma_A=\sigma_B=\sigma_C=\frac{1-\beta}{\beta}$ and $\gamma=1$ in Theorem~\ref{th:ryu}.\qedhere
\end{proof}

In order to examine the performance of the algorithm~\eqref{eq:Ryu_best_approx} and the effect of the parameters $\beta$ and $\lambda$, we considered the problem of finding the nearest positive semidefinite doubly stochastic matrix with prescribed entries. Denoting by $\Omega$ the location of the entries that are prescribed and by $M$ the matrix with its values, this problem can be described in terms of three sets:
\begin{align*}
C_1&:=\{X\in\R^{n\times n}: Xe=X^Te=e\},\\
C_2&:=\{X\in\R^{n\times n}: X_{ij}\geq 0\text{ for }i,j=1,\ldots,n\text{ and }X_{ij}=M_{ij}\text{ for all }(i,j)\in\Omega\},\\
C_3&:=\{X\in\R^{n\times n}: X\text{ is positive semidefinite}\},
\end{align*}
where $e=[1,1,\ldots,1]^T\in\R^n$. The projectors onto these sets have a closed form:
\begin{align*}
P_{C_1}(X)&=(I-J)X(I-J)+J,\text{ where }J=ee^T/n,\\
P_{C_2}(X)_{ij}&=\left\{\begin{array}{ll}
M_{ij} &\text{if }(i,j)\in\Omega,\\
\max\{X_{ij},0\}&\text{otherwise,}
\end{array}\right.\\
P_{C_3}(X)&=\frac{Y+P}{2},\text{ where }Y=\frac{X+X^T}{2}\text{ and } Y=UP\text{ is a polar decomposition},
\end{align*}
for $X=(X_{ij})\in\R^{n\times n}$; see \cite[Proposition~4.4]{takouda2005probleme} for the formula for $P_{C_1}$ and \cite[Theorem~2.1]{higham1988computing} for the formula $P_{C_3}$. For further details and extensions, see \cite[Section~3]{aragon2014matrix} and \cite{bauschke2021projecting}.

\begin{figure}[!htbp]
\centering
     \includegraphics[width=.78\textwidth]{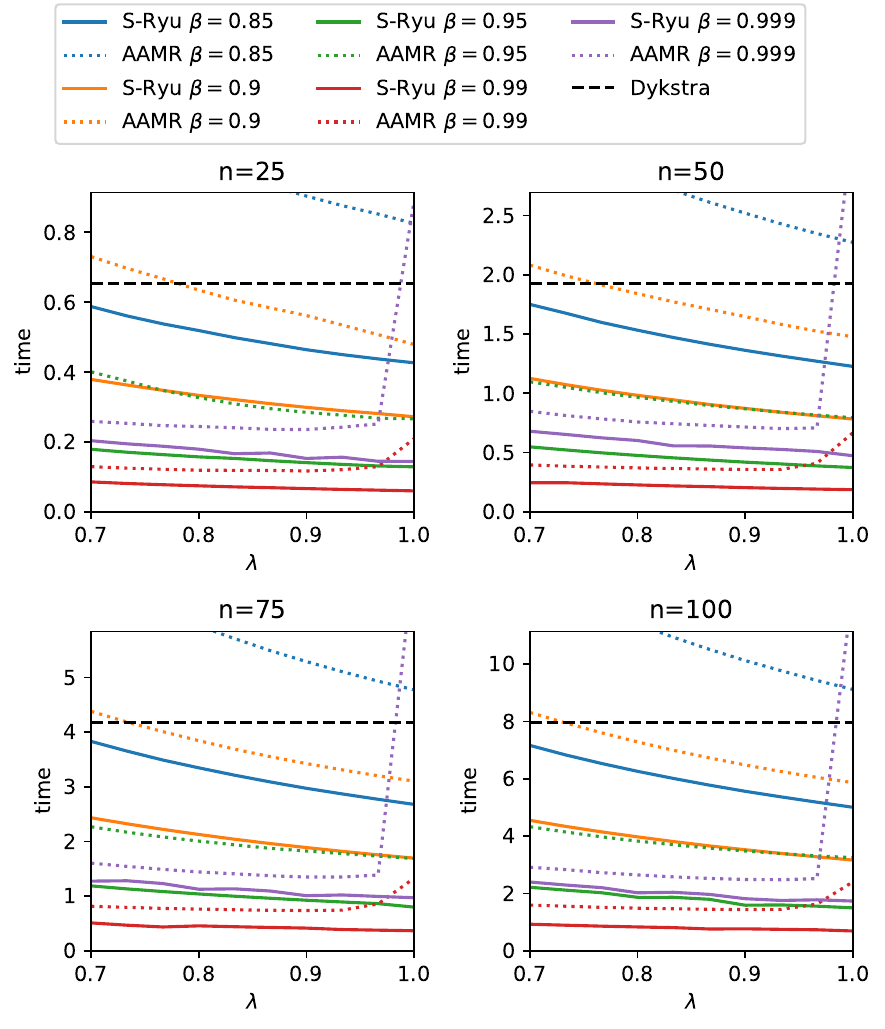}
  \caption{Comparison of the performance of Dykstra~\cite{boyle1986method}, AAMR~\cite{aragon2018new} and S-Ryu~\eqref{eq:Ryu_best_approx} for finding the nearest positive semidefinite doubly stochastic matrix with a prescribed entry. For each pair of parameters $(\beta,\lambda)$ we represent the average time of 20 random instances.\label{fig:matrix_comparison}}
\end{figure}

In our test, we took $\Omega=\{(1,1)\}$ with $M_{11}=0.25$ (that is, we prescribed the first entry to $0.25$). We compared the performance of S-Ryu against Dykstra's method~\cite{boyle1986method} and AAMR~\cite{aragon2018new}. To this aim, we computed the nearest matrix satisfying the constraints to a symmetric matrix with random entries uniformly distributed in $(-2,2)$.
In order to apply Corollary~\ref{cor:ryu_3sets}, for S-Ryu, or \cite[Theorem~5.1]{aragon2018new}, for AAMR, we need to check that the strong CHIP condition holds for $C_1$, $C_2$ and $C_3$ (see Remark~\ref{r:sCHIP}). A sufficient condition is the nonempty intersection of the relative interiors of the three sets (see, e.g., \cite[Corollary 23.8.1]{Rock72}). Since $C_1$ and $C_2$ are affine subspaces, and the relative interior of $C_3$ consists of the positive definite matrices in $X$ (see, e.g.,~\cite[Exercise~5.12]{Berman94}), it suffices to find a positive definite matrix in $C_1\cap C_2$. A possible choice is the matrix
$$M:=\frac{0.25n-1}{n-1}I+\frac{0.75}{n-1}ee^T\in C_1\cap C_2.$$
The matrix $M$ is clearly symmetric and it can be readily checked that its eigenvalues are $1$ and $\frac{0.25n-1}{n-1}$ (with multiplicity $n-1$). Hence, $M$ is positive definite as long as $n\geq 5$, which holds for the instances considered here.

For the algorithm implementations in our first test, we took $10$ values of $\lambda$ equispaced in $[0.7,1]$ and $\beta\in\{0.85,0.9,0.95,0.99,0.999\}$. For each pair of values and each $n\in\{25,50,75,100\}$ we generated $20$ random matrices in $\R^{n\times n}$. We have represented in Figure~\ref{fig:matrix_comparison} the average time required by each of the algorithms to achieve $\sum_{i=1}^3\|U_k-P_{C_i}(U_k)\|\leq 10^{-5}$ for each pair of parameters. For brevity, we do not include the figures with the iteration count because they produce a similar result. In these numerical results, the fastest algorithm was S-Ryu with parameters $(\beta,\lambda)=(.99,1)$.

We performed a second test for larger matrices, where we fixed the value of $(\beta,\lambda)$ to
$(0.99,1)$ for S-Ryu and $(0.99,0.95)$ for AAMR. The results are
summarised in Figure~\ref{fig:matrix_comparison_large}. We observe that S-Ryu was consistently 10 times
faster than Dykstra and more than 2 times faster than AAMR.

\begin{figure}[!htbp]
\centering
  \begin{subfigure}[c]{0.48\textwidth}
     \includegraphics[width=\textwidth]{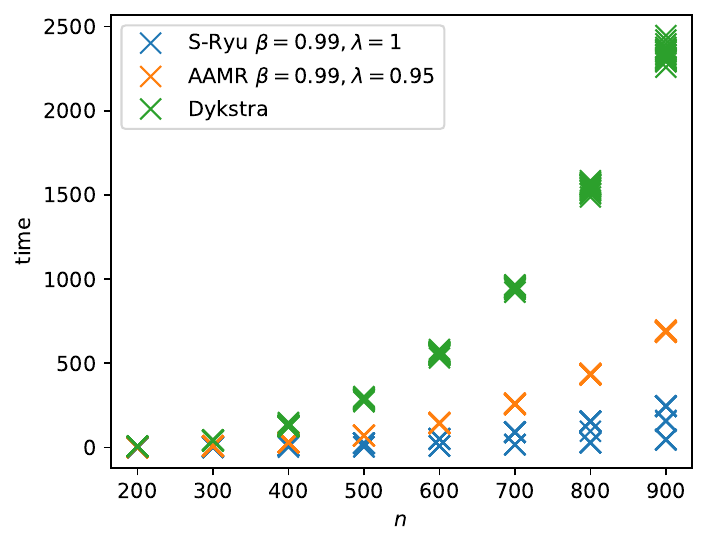}
     \subcaption{Time.}
  \end{subfigure}
  \hfill
  \begin{subfigure}[c]{0.48\textwidth}
     \includegraphics[width=\textwidth]{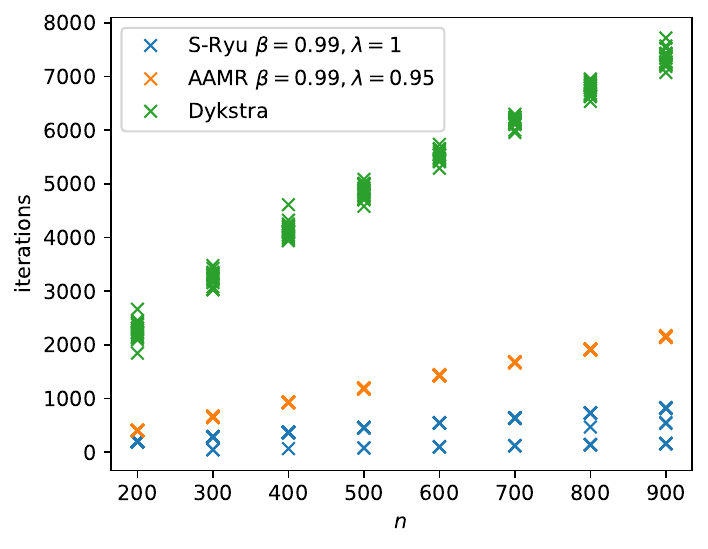}
     \subcaption{Iterations.}
  \end{subfigure}\\
  \begin{subfigure}[c]{0.48\textwidth}
     \includegraphics[width=\textwidth]{{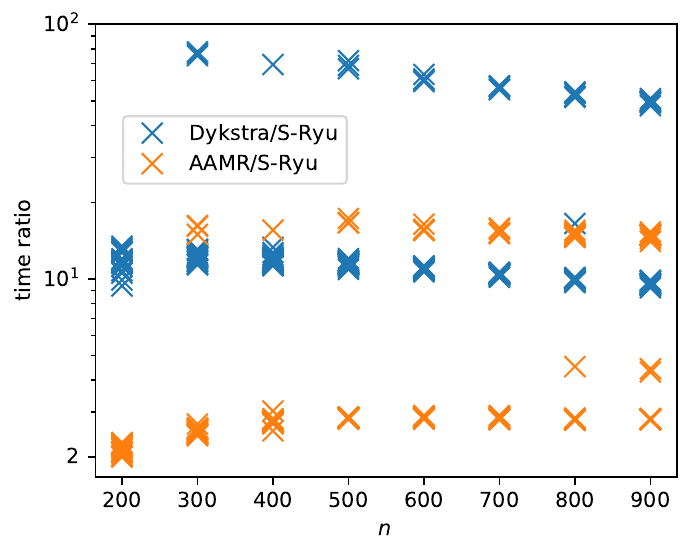}}
     \subcaption{Time ratio (log.\ scale).}
  \end{subfigure}
  \hfill
  \begin{subfigure}[c]{0.48\textwidth}
     \includegraphics[width=\textwidth]{{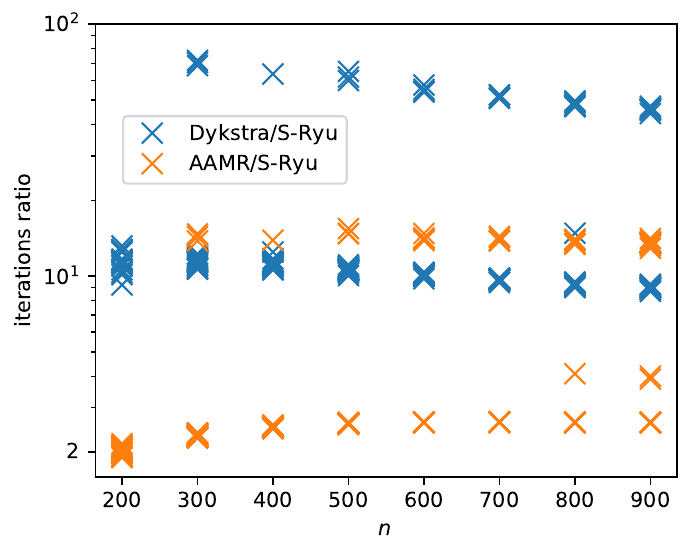}}
     \subcaption{Iterations ratio (in log. scale).}
  \end{subfigure}
  \caption{Comparison of the performance of Dykstra \cite{boyle1986method}, AAMR~\cite{aragon2018new} and S-Ryu~\eqref{eq:Ryu_best_approx} for finding the nearest positive semidefinite doubly stochastic matrix with a prescribed entry on 20 random instances. We represent the time and the iterations required by each algorithm with respect to the size, as well as the ratios.\label{fig:matrix_comparison_large}}
\end{figure}

\subsection{ROF-type Models for Image Denoising}
Let $\phi\colon\Hilbert'\to{]-\infty,+\infty]}$ and $g\colon\Hilbert\to{]-\infty,+\infty]}$ be proper, lsc and convex and let $K\colon\Hilbert\to\Hilbert'$ be a bounded linear operator. Given $q\in\Hilbert$ and $\eta>0$, consider the problem
\begin{equation}\label{eq:rof-proto}
 \min_{x\in\Hilbert}\frac{\eta}{2}\|x-q\|^2+\phi(Kx)+g(x),
\end{equation}
which is equivalent to computing the proximity operator of $\phi\circ K+g$ (with parameter $1/\eta$) at $q$. Using the identity $\phi(Kx)=\phi^{**}(Kx)=\sup_{y\in\Hilbert'}\{\langle Kx,y\rangle-\phi^*(y)\}$, problem~\eqref{eq:rof-proto} may be expressed in saddle-point form as
 $$ \min_{x\in\Hilbert}\sup_{y\in\Hilbert'}\frac{\eta}{2}\|x-q\|^2+\langle Kx,y\rangle-\phi^*(y)+g(x). $$
Solutions to this problem (in the sense of saddle-points \cite[p.~244]{rockafellar1970monotone}) can be characterised by the operator inclusion
$$ \binom{q}{0}
 \in \left( \begin{pmatrix}\Id&0\\0&\Id\\ \end{pmatrix}
            +\binom{\frac{1}{\eta}\partial g}{\partial\bigl(\frac{1}{\eta}\phi^*-\frac{1}{2}\|\cdot\|^2\bigr)}
            +\frac{1}{\eta}\begin{pmatrix} 0 & K^* \\ -K & 0 \\ \end{pmatrix}\right)\binom{x}{y}$$
which is equivalent to $\binom{x}{y} \in J_{A+B}\left(\binom{q}{0}\right)$ where the operators $A$ and $B$ are given by
\begin{equation}\label{eq:rof operators}
A:=\binom{\frac{1}{\eta}\partial g}{\partial\bigl(\frac{1}{\eta}\phi^*-\frac{1}{2}\|\cdot\|^2\bigr)},\quad B:=\frac{1}{\eta}\begin{pmatrix} 0 & K^* \\ -K & 0 \\ \end{pmatrix}.
\end{equation}

Here we note that the operator $A$ is maximally $\alpha_A$-monotone and $B$ is $\alpha_B$-monotone with $\alpha=(-1,0)$. Thus, by choosing $\sigma_B=0$ and $\theta=\sigma_A>0$, we have $\theta\alpha_A+\sigma_A=\theta\alpha_B+\sigma_B=0$ and so we can ensure that conditions in~\eqref{a1_rel} are satisfied. Hence, in view of Remark~\ref{r:fbf}, the strengthened Tseng's method (S-Teng) in Theorem~\ref{th:fbf} converges weakly assuming that a saddle-point for the problem exists.

\begin{figure}[!htbp]
\centering
 \begin{tabular}{ccc}
   $\eta=1$ & $\eta=4$ & $\eta=8$ \\
   \includegraphics[width=3cm]{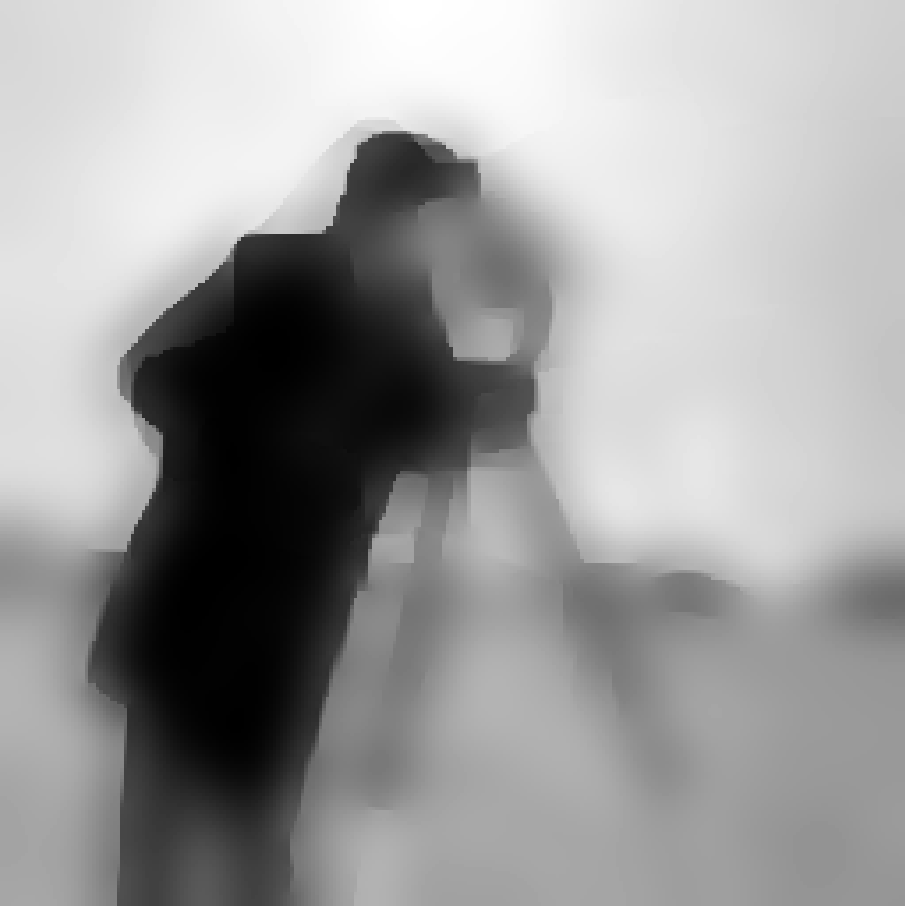}
   & \includegraphics[width=3cm]{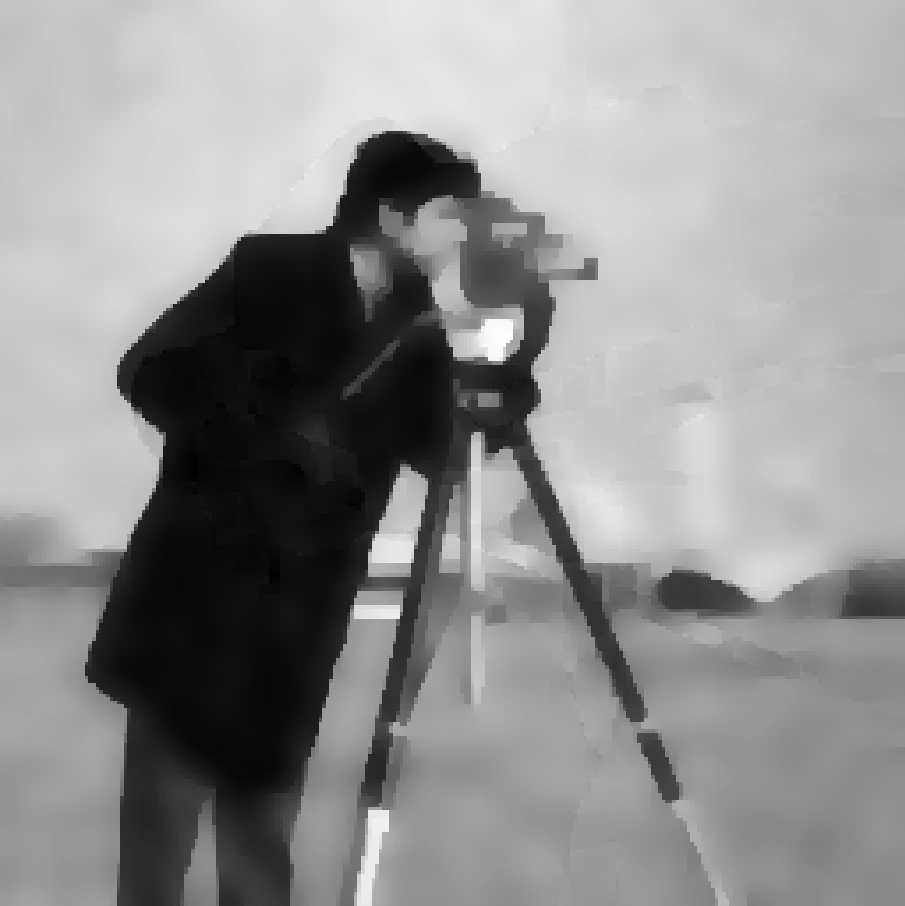}
   & \includegraphics[width=3cm]{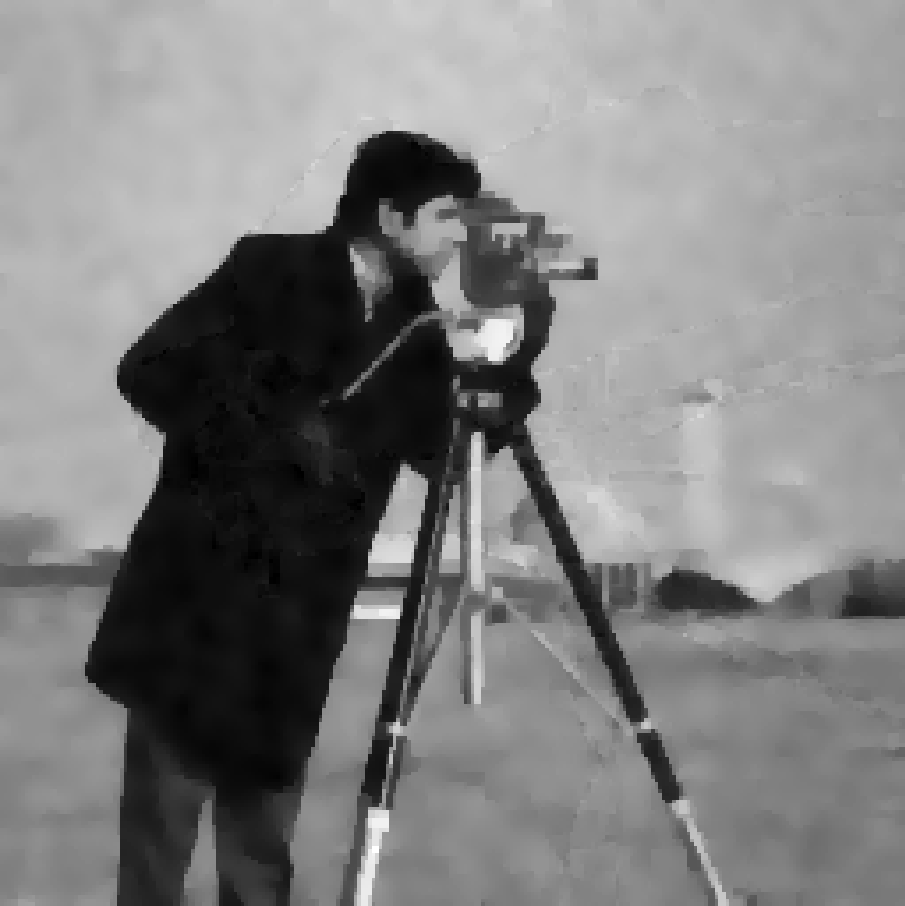}  \\[0.5ex]
   $\eta=12$ & $\eta=16$ & $\eta=20$ \\
   \includegraphics[width=3cm]{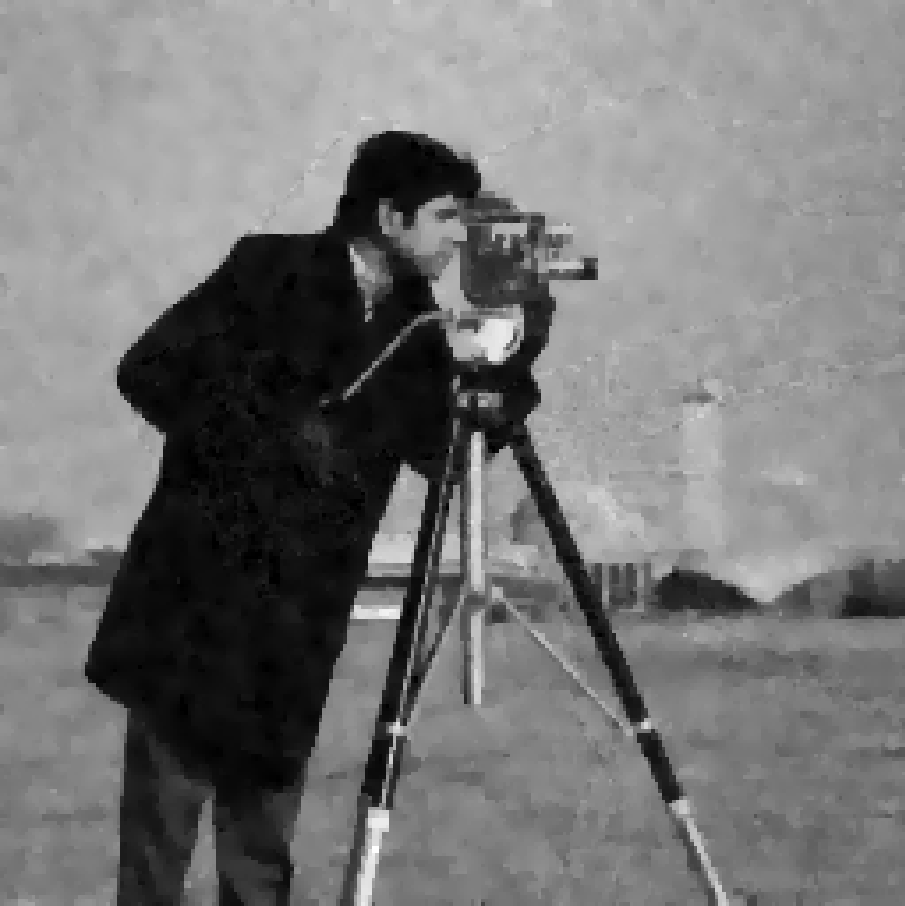}
   & \includegraphics[width=3cm]{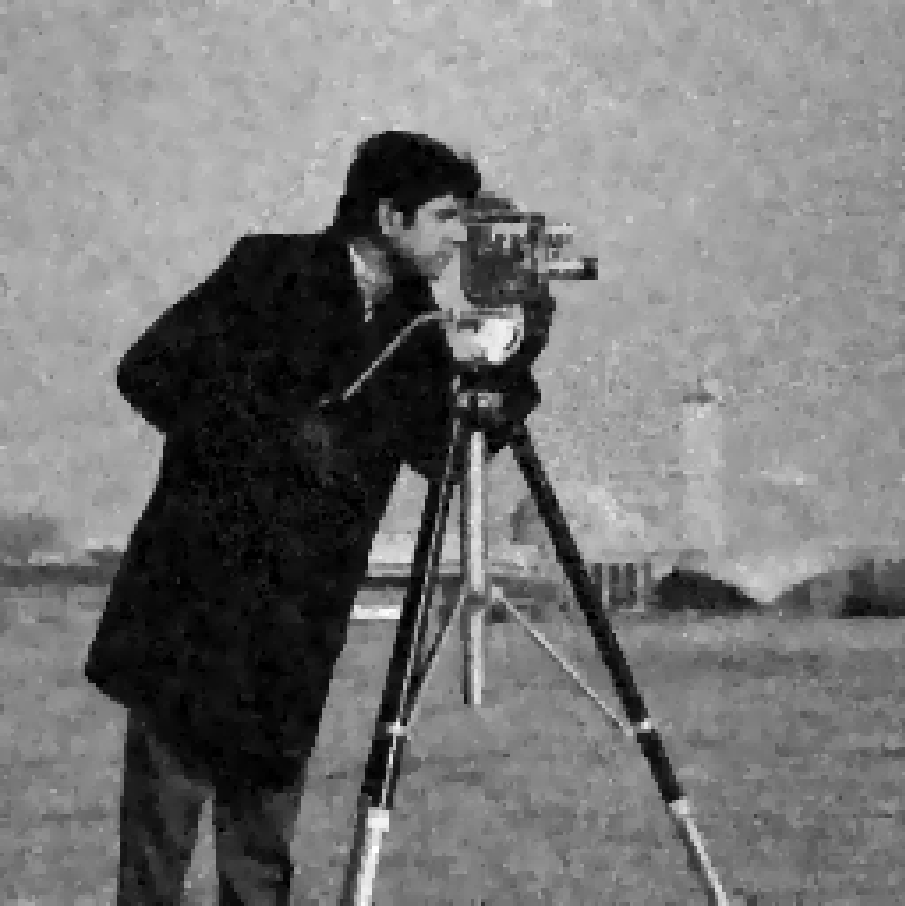}
   & \includegraphics[width=3cm]{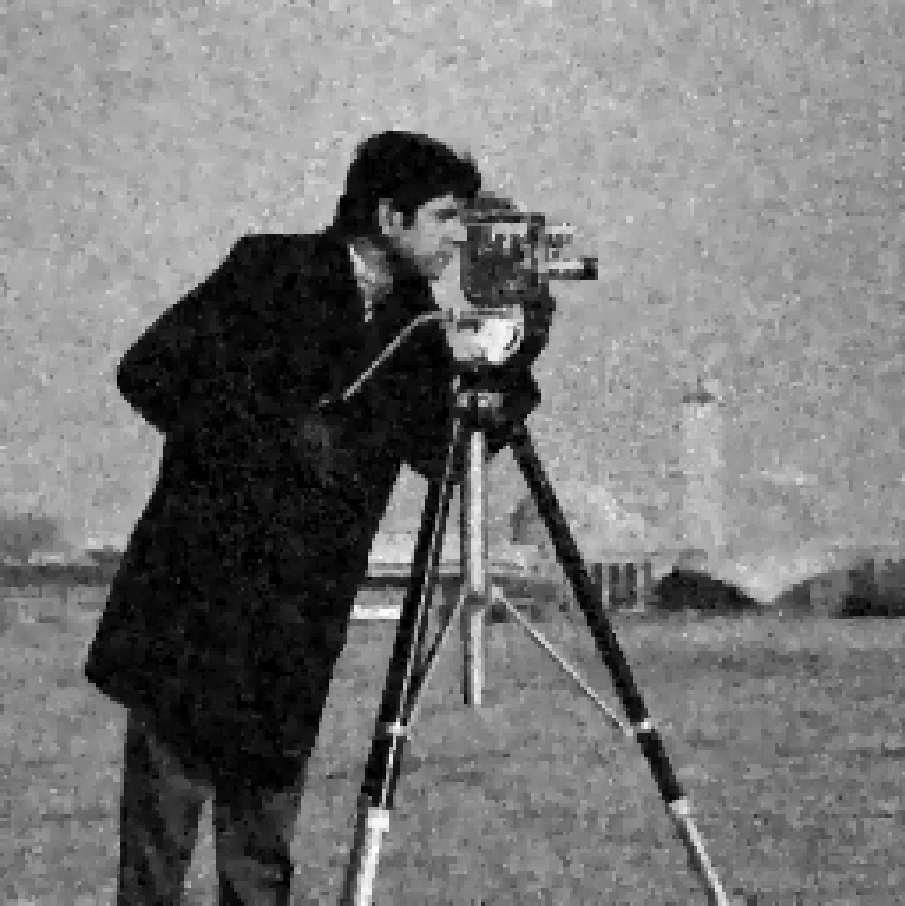} \\
 \end{tabular}
\caption{The effect of $\eta$ on the solution of \eqref{eq:rof-proto}.}\label{fig:reg param}
\end{figure}

Let $\Hilbert=\mathbb{R}^{n\times n}$ represent the $n\times n$ pixel grayscale images (with pixel values in $[0,1]$) and let $\Hilbert'=\mathbb{R}^{n\times n}\times\mathbb{R}^{n\times n}$. Then the \emph{Rudin--Osher--Fatemi (ROF) model} \cite{ROFmodel} for image denoising can be understood as a particular case of \eqref{eq:rof-proto} where $\phi(Kx)$ represents the discrete version of the isotropic TV-norm and $g=0$. In this setting, the $K\colon\Hilbert\to\Hilbert'$ denotes the \emph{discrete gradient} with Lipschitz constant $\sqrt{8}$ and $\phi\bigl(\binom{y^1}{y^2}\bigr)=\sum_{i,j=1}^n\sqrt{ (y^1_{i,j})^2+(y^2_{i,j})^2 }$. For further details, see \cite{chambolle2011first}. A setting with $g\neq 0$ was considered in \cite{chen2019iterative} where it was taken as $g=\iota_C$ for a convex constraint set $C$. The addition of this constraint set allows a priori information about the image to be incorporated. In this work, we take $C = \{x\in\mathbb{R}^{n\times n}:0\leq x\leq 1\}$ to encode the bounds on legal pixel values. Also note that, since $\phi$ is continuous on $\Hilbert'$, the assumptions of Theorem~\ref{th:primaldual} hold (see also Remark~\ref{r:primaldual}).

In order to examine the effect of the varying algorithm parameters on performance, we applied the algorithms presented in Theorems~\ref{th:fbf} and \ref{th:primaldual} to the ROF-denoising model with the constraint $C$ via the operator formulation provided by \eqref{eq:rof operators} and \eqref{eq:rof-proto}, respectively. The regularisation parameter was chosen by trial and error to be $\eta=12$ (see Figure~\ref{fig:reg param}). The noisy image to be denoised is given by $q\in\Hilbert$. For all tests, $\binom{q}{0}\in\Hilbert\times\Hilbert'$ was used as the initial point (i.e., the noisy image was used as the initialisation).

\begin{figure}[!htbp]
\centering
     \includegraphics[height=4cm]{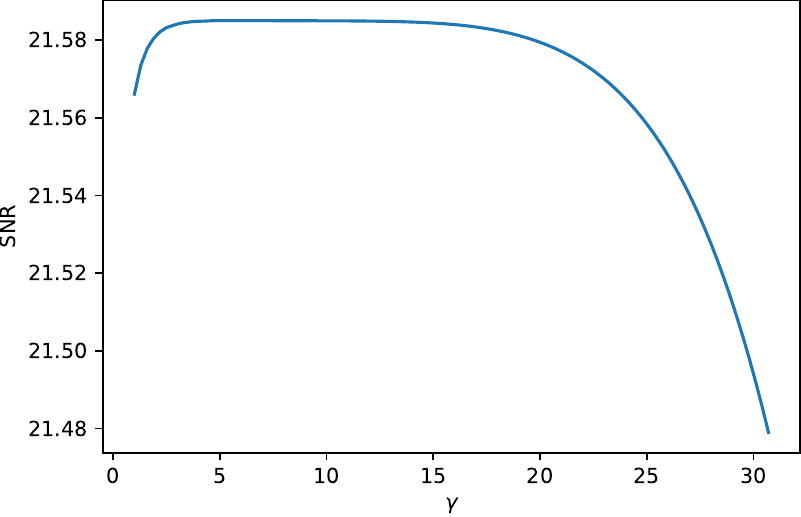} \hspace{1cm}
     \includegraphics[height=4cm]{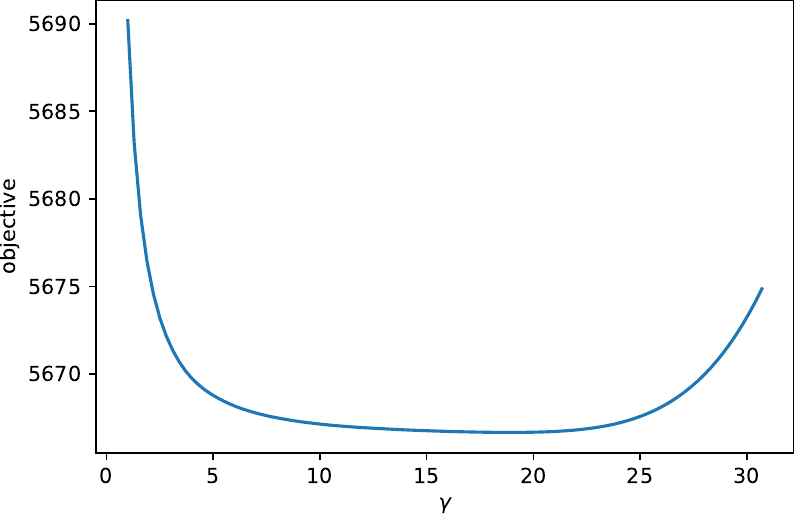}

%
   \caption{The effect of the dual stepsize, $\gamma$, for S-PD after $100$ iterations with the primal stepsize taken to be $\tau=\frac{0.99}{8\gamma}$.\label{f:rof1}}
\end{figure}

Figure~\ref{f:rof1} shows the effect on the signal to noise ratio (SNR) and objective function value after 100 iterations. For the strengthened primal-dual method (S-PD), we examined the effect of changing the dual stepsize, denoted by $\gamma$, with  the primal stepsize chosen as $\tau=\frac{0.99}{\gamma \|K\|^2}=\frac{0.12375}{\gamma}$ so that the condition $\gamma\tau \|K\|^2<1$
holds. The figure suggests $\gamma=15$ as a good choice for S-PD.  For the strengthened Tseng's method (S-Tseng), the stepsize $\gamma$ was chosen to satisfy $\gamma=\frac{0.99}{\theta\kappa+\sigma_B}$ where $\kappa=\sqrt{8}/\eta$ denotes the Lipschitz constant of the operator $B$. This is equivalent to asserting that $\gamma\sigma_A=\frac{11.88}{\sqrt{8}}$.

\begin{figure}[!htbp]
\centering
  \begin{tabular}{lccc@{\hskip 2em}lc}
   \rotatebox[origin=c]{90}{S-PD}
     & \raisebox{-0.5\height}{\includegraphics[width=2.8cm]{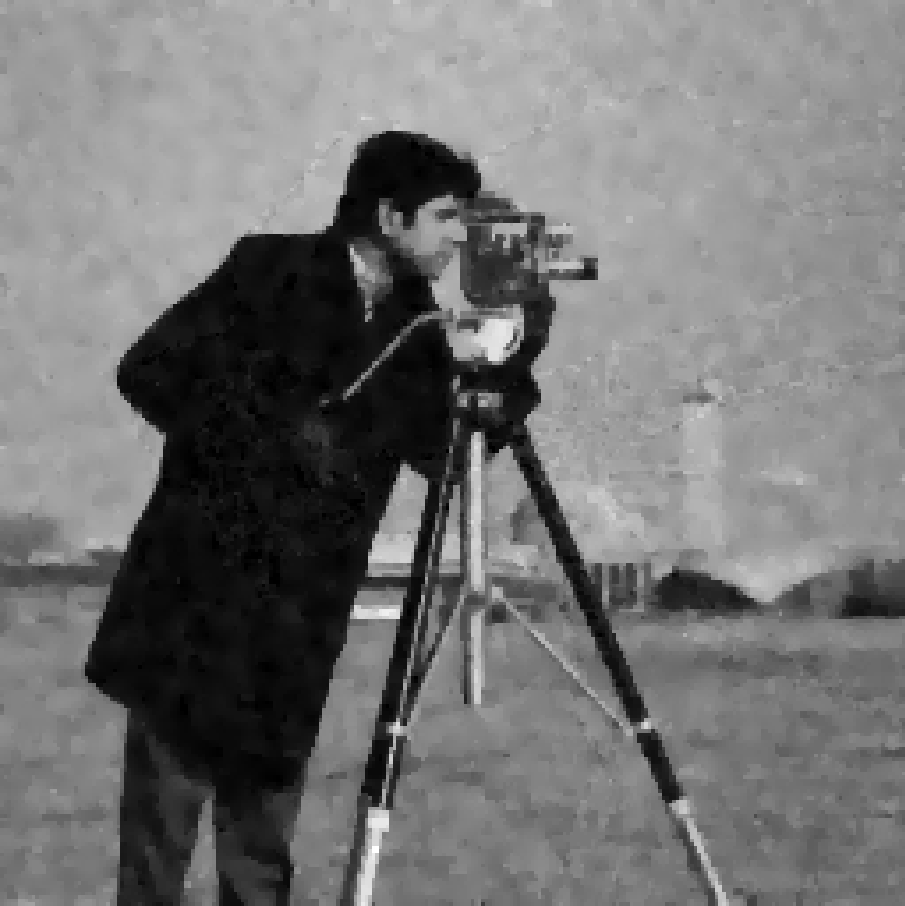}}
     & \raisebox{-0.5\height}{\includegraphics[width=2.8cm]{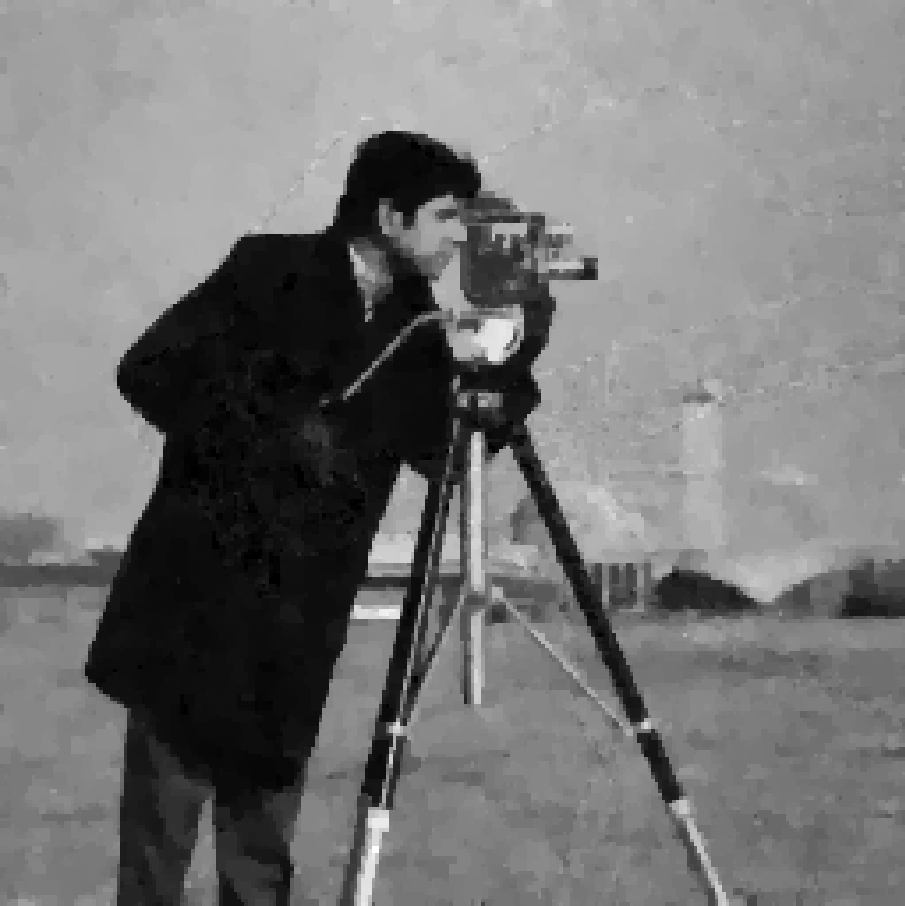}}
     & \raisebox{-0.5\height}{\includegraphics[width=2.8cm]{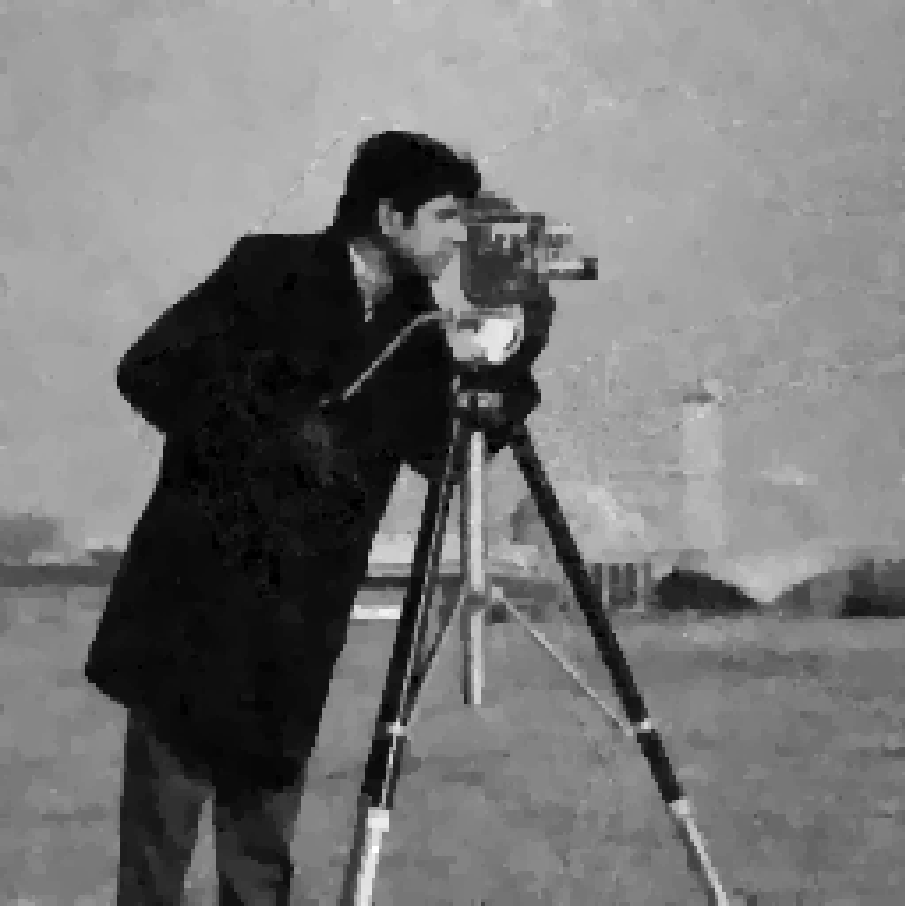}}
     & \rotatebox[origin=c]{90}{Original}
     & \raisebox{-0.5\height}{\includegraphics[width=2.8cm]{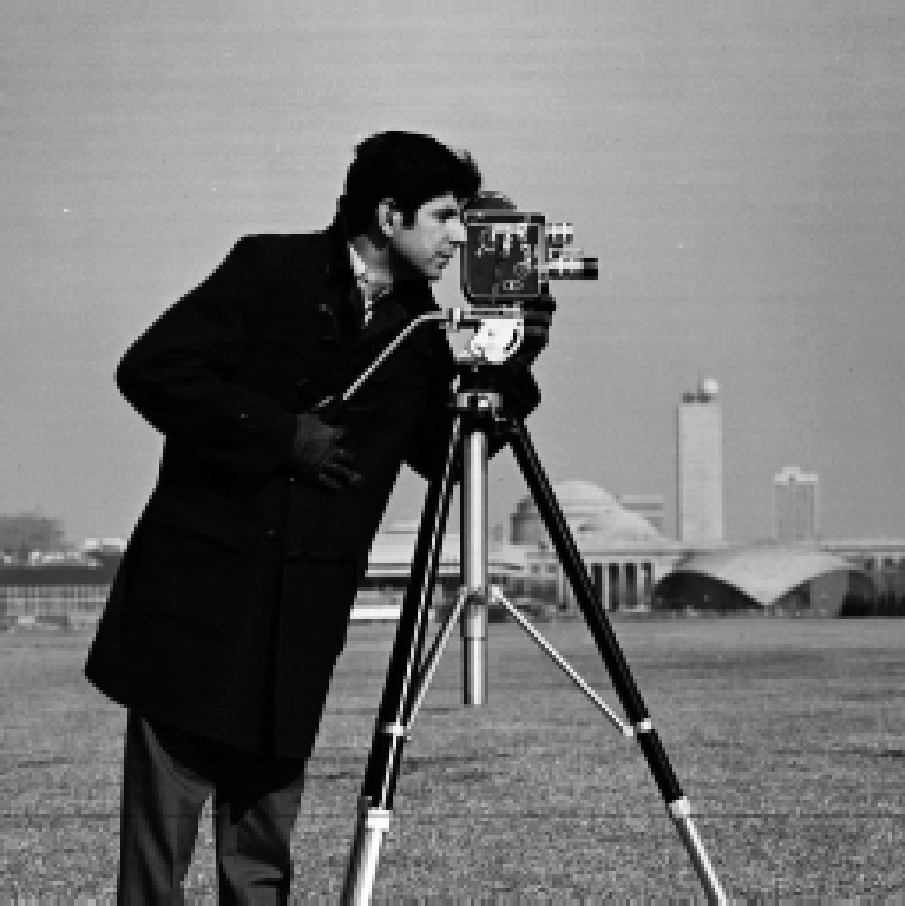}} \\[1.7cm]
   \rotatebox[origin=c]{90}{S-Tseng}
     & \raisebox{-0.5\height}{\includegraphics[width=2.8cm]{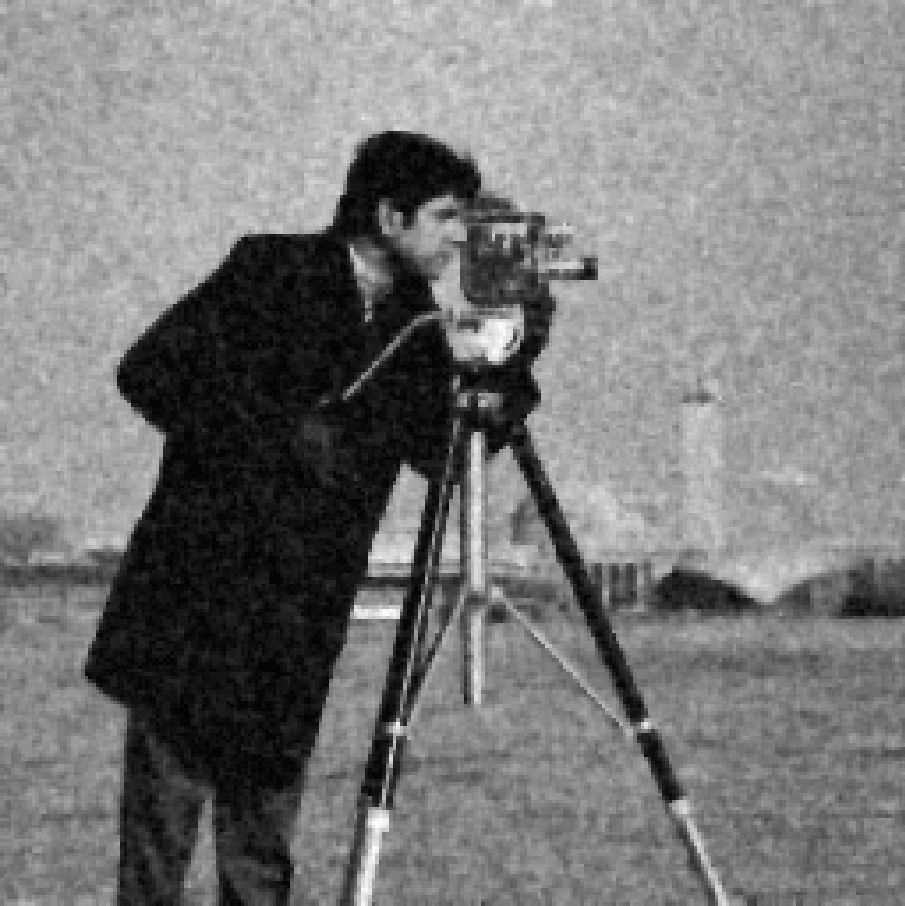}}
     & \raisebox{-0.5\height}{\includegraphics[width=2.8cm]{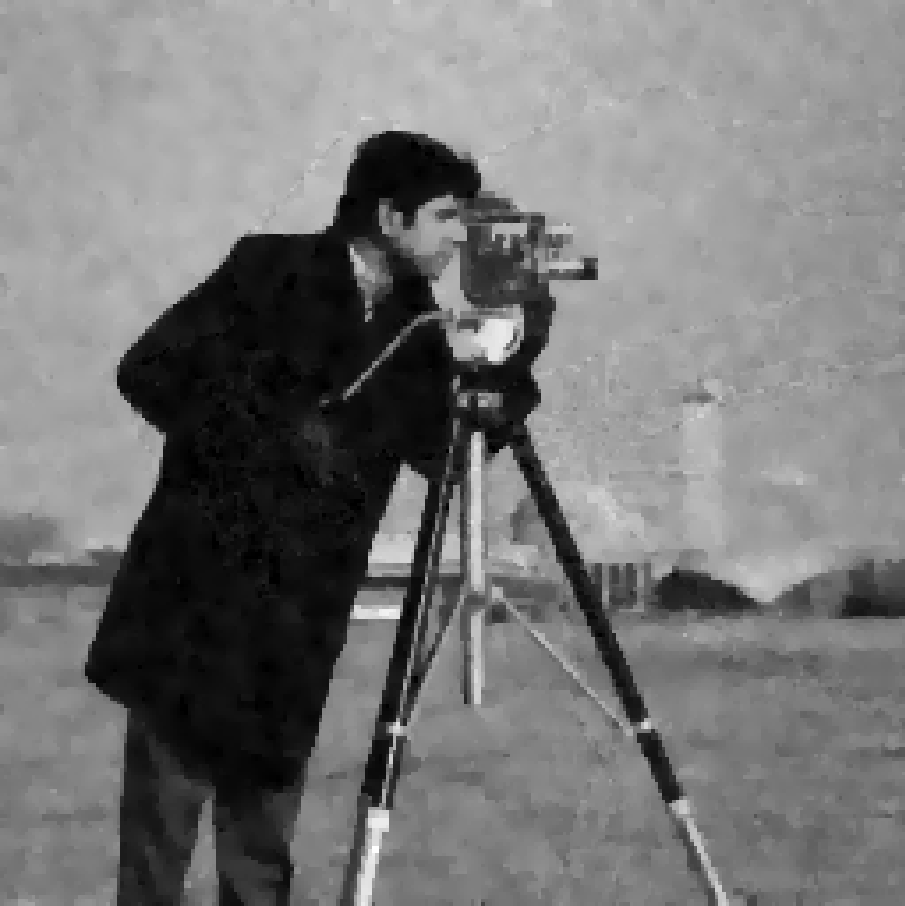}}
     & \raisebox{-0.5\height}{\includegraphics[width=2.8cm]{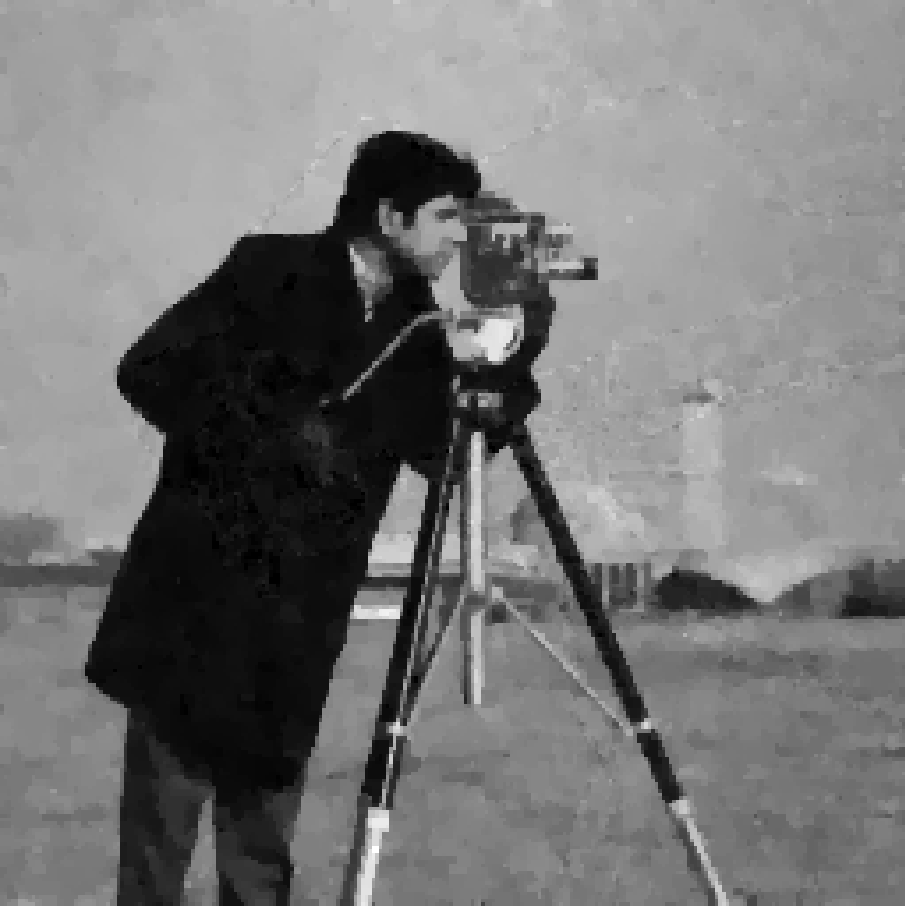}}
     & \rotatebox[origin=c]{90}{Noisy}
     & \raisebox{-0.5\height}{\includegraphics[width=2.8cm]{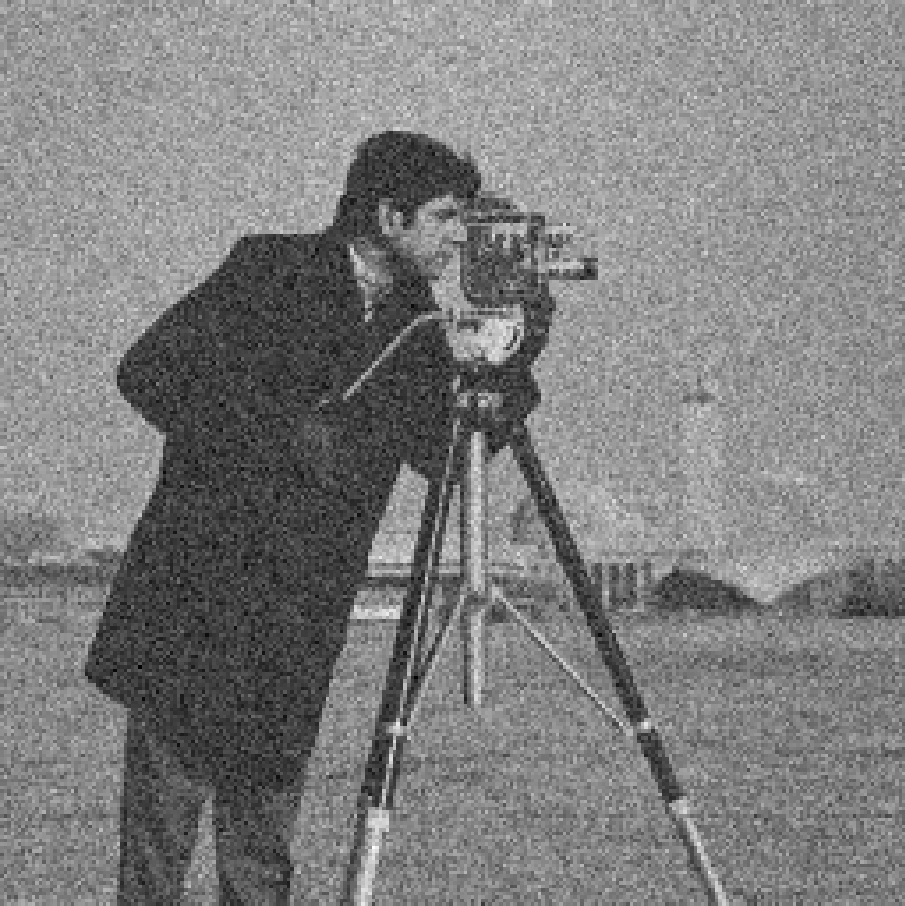}} \\
     & 10 iterations & 100 iterations & 1000 iterations & \\
  \end{tabular}

  \caption{Results after $10,100,1000$ iterations for the algorithms from Section~\ref{s:fb methods} with $\eta=12$, $\gamma=15$, $\tau=\frac{33}{4000}$, $\theta=10.1$ and $\sigma=(10.0,0.1)$. The image (cameraman), with and without additive Gaussian noise, is shown on the right.\label{f:rof2}}
\end{figure}

Figure~\ref{f:rof2} also suggests that the strengthened primal-dual method (S-PD)  performs slightly better than the strengthened Tseng's method (S-Tseng). Further computational results for the strengthened primal-dual method applied to four square test images with $n\in\{250,500,750,1000\}$ are shown in Figure~\ref{f:rof3}. The final change iterates (i.e., $\|\binom{x_k}{y_k}-\binom{x_{k-1}}{y_{k-1}}\|$), signal-to-noise ratio (SNR), final objective function value, and CPU after $k=100$ iterations can be found in Table~\ref{t:rof}. Figure~\ref{f:rof3} also shows a comparison with Chen \& Tang's method using the best performing algorithms parameters as described in \cite[Section~5]{chen2019iterative}. The results suggest that the performance of both methods is similar.

\begin{table}[!htb]
\caption{Results for $\eta=12$ after $100$ iterations of (left values) S-PD with $\gamma=15$ and (right values) Chen \& Tang's method from \cite{chen2019iterative}. Time is reported as the average from ten replications.\label{t:rof}}
\centering
\begin{tabular}{|l|r|rr|rr|rr|rr|} \hline
            & $n$    &  \multicolumn{2}{c|}{Change}  &  \multicolumn{2}{c|}{SNR}   & \multicolumn{2}{c|}{Objective Function} &   \multicolumn{2}{c|}{Time (s)} \\ \hline
Alicante    & 250    &  0.11 & 0.01    &  22.73 & 22.74    &   4 542.18 &   4 547.02   &  0.27 &  0.27 \\
G\"ottingen & 500    &  0.19 & 0.02    &  17.65 & 17.66    &  18 860.98 &  18 878.44   &  1.36 &  1.48 \\
Sydney      & 750    &  0.26 & 0.02    &  21.50 & 21.50    &  51 645.81 &  51 677.34   &  3.71 &  4.06 \\
Boston      & 1000   &  0.19 & 0.02    &  18.32 & 18.32    & 114 557.63 & 114 589.94   &  7.63 & 8.31 \\ \hline
\end{tabular}

\end{table}

\begin{figure}[!htb]
\begin{center}
  \begin{tabular}{lcccc}
  & Alicante & G\"ottingen & Sydney  & Boston \\[1ex]
   \rotatebox[origin=c]{90}{Original}
     & \raisebox{-0.5\height}{\includegraphics[width=2.8cm]{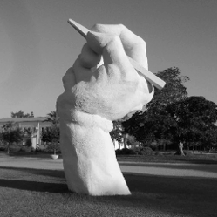}}
     & \raisebox{-0.5\height}{\includegraphics[width=2.8cm]{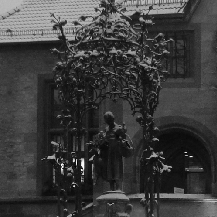}}
     & \raisebox{-0.5\height}{\includegraphics[width=2.8cm]{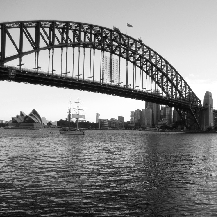}}
     & \raisebox{-0.5\height}{\includegraphics[width=2.8cm]{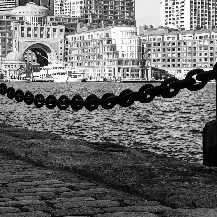}} \\[1.7cm]
   \rotatebox[origin=c]{90}{\parbox{2.7cm}{\centering Recovered\\ (S-PD)}}
     & \raisebox{-0.5\height}{\includegraphics[width=2.8cm]{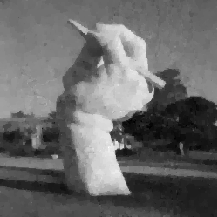}}
     & \raisebox{-0.5\height}{\includegraphics[width=2.8cm]{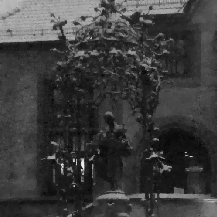}}
     & \raisebox{-0.5\height}{\includegraphics[width=2.8cm]{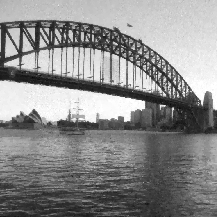}}
     & \raisebox{-0.5\height}{\includegraphics[width=2.8cm]{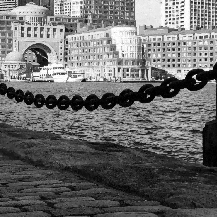}} \\[1.7cm]
   \rotatebox[origin=c]{90}{\parbox{2.7cm}{\centering Recovered\\ (Cheng--Tang)}}
     & \raisebox{-0.5\height}{\includegraphics[width=2.8cm]{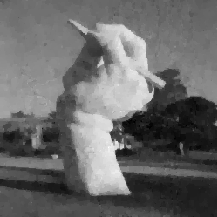}}
     & \raisebox{-0.5\height}{\includegraphics[width=2.8cm]{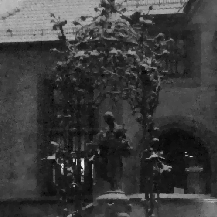}}
     & \raisebox{-0.5\height}{\includegraphics[width=2.8cm]{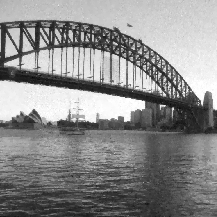}}
     & \raisebox{-0.5\height}{\includegraphics[width=2.8cm]{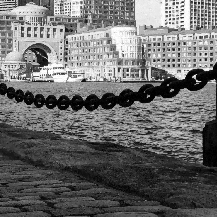}} \\[1.7cm]
   \rotatebox[origin=c]{90}{Noisy}
     & \raisebox{-0.5\height}{\includegraphics[width=2.8cm]{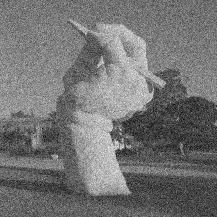}}
     & \raisebox{-0.5\height}{\includegraphics[width=2.8cm]{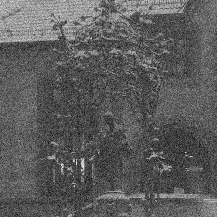}}
     & \raisebox{-0.5\height}{\includegraphics[width=2.8cm]{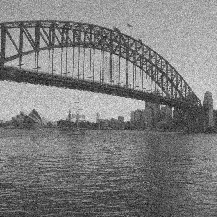}}
     & \raisebox{-0.5\height}{\includegraphics[width=2.8cm]{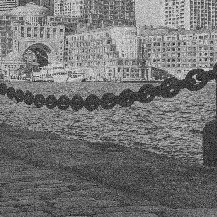}}\\
  \end{tabular}
\end{center}
\caption{Original, noisy and typical de-noised images for the  results reported in Table~\ref{t:rof}.\label{f:rof3}}
\end{figure}

\subsection{PDEs with Partially Blinded Laplacians}
In this subsection, we consider a problem in elliptic PDEs previously studied in~\cite[Section~5.2]{adly2019decomposition}. To this end, let $\Hilbert=\mathcal{L}^2(\Omega)$ denote the Hilbert space of real (Lebesgue) square-integrable functions defined on a nonempty open subset $\Omega\subset\R^d$ with a nonempty, bounded and Lipschitz boundary denoted by $\partial\Omega$. Let $L_+^2(\Omega)=\{v\in L^2(\Omega):v\geq 0\text{~a.e.\ on~}\Omega\}$ and $L^2_\Delta=\{v\in L^2(\Omega):\Delta v\in L^2(\Omega)\}$, where $\Delta$ denotes the \emph{Laplace operator}. Further, let $H^1_0(\Omega):=\{v\in H^1(\Omega):v_{|_{\partial\Omega}}=0\}$, where $H^1(\Omega)$ denotes the standard Sobolev space of functions having first order distribution derivatives in $L^2(\Omega)$ and $v_{|_{\partial\Omega}}$ denotes the \emph{trace} of $v$ on $\partial\Omega$ (see, for instance, \cite{attouch2014variational,brezis2011}). Given $u\in\Hilbert$, we denote $u^+:=\max\{u,0\}$ and $u^-:=\min\{u,0\}$ which are understood in the pointwise sense.

Let $f\in L^2(\Omega)$ be fixed. In this subsection, we consider the \emph{partially blinded problem} with homogeneous Dirichlet boundary condition
\begin{equation}\label{eq:PBP0}
\text{find~}u\in L^2(\Omega)\text{~such that~}u^+\in H^1_0(\Omega)\cap L^2_\Delta(\Omega)\text{~and~}-\Delta(u^+)+u=f, \tag{PBP$_0$}
\end{equation}
as well as the corresponding \emph{obstacle problem} given by
\begin{equation}\label{eq:OP0}
\text{find~}v\in L^2(\Omega)\text{~such that~}v\in H^1_0(\Omega)\cap L^2_\Delta(\Omega)\text{~and~}0\leq(-\Delta v+v-f)\perp v\geq 0.\tag{OP$_0$}
\end{equation}
As explained in \cite{adly2019decomposition}, \eqref{eq:PBP0} derives its name from the fact that the Laplacian operator is \emph{partially blinded} in the sense that diffusion only occurs on the nonnegative part of the unknown function $u$.

The following proposition collects results useful for solving these two problems.
\begin{proposition}\label{prop:pde}
Let $A:=N_{L^2_+(\Omega)}$ and let $B:=-\Delta:H^1_0(\Omega)\cap L^2_{\Delta}(\Omega)\to L^2(\Omega)$. Then:
\begin{enumerate}[(i)]
\item\label{it:pde_i} $A$, $B$ and $A+B$ are maximally monotone on $L^2(\Omega)$.
\item \label{it:pde_ii} Let $g\in L^2(\Omega)$ and let $\gamma>0$. Then $J_{\gamma A}(g)=g^+$ and $J_{\gamma B}(g)$ is the (unique) solution of linear boundary value problem given by
\begin{equation}\label{eq:LP0}
\text{find~}w\in L^2(\Omega)\text{~such that~}w\in H^1_0(\Omega)\cap L^2_\Delta(\Omega)\text{~and~}-\Delta w+\frac{1}{\gamma}w=\frac{1}{\gamma}g.\tag{LP$_0$}
\end{equation}
\item\label{it:pde_iii} The unique solution of \eqref{eq:PBP0} is given by $u:=\bigl(\Id-B\circ J_{A+B}\bigr)(f)$.
\item\label{it:pde_iv} The unique solution of \eqref{eq:OP0} is given by $v:=J_{A+B}(f)$.
\item\label{it:pde_v} The functions $u$ and $v$ satisfy $v=u^+$.
\end{enumerate}
\end{proposition}
\begin{proof}
Items \eqref{it:pde_i}, \eqref{it:pde_ii}, \eqref{it:pde_iv} and \eqref{it:pde_v} appear in the proof of \cite[Proposition~5]{adly2019decomposition} as well as the claimed uniqueness of the solution to \eqref{eq:PBP0} in \eqref{it:pde_iii}. The formula for $u$ follows by combining \eqref{it:pde_iv}, \eqref{it:pde_v} and \eqref{eq:PBP0}.
\end{proof}

Proposition~\ref{prop:pde} shows that to solve either \eqref{eq:PBP0} or \eqref{eq:OP0} it suffices to compute the resolvent of $A+B$ at $f$. Since the resolvent of $A$ and $B$ are both accessible, according to Proposition~\ref{prop:pde}\eqref{it:pde_ii}, the strengthened Douglas--Rachford method (S-DR, see Theorem~\ref{th:DR}) can be applied with $\theta=\sigma_A+\sigma_B$. Observe that the assumption $f\in\ran(\Id+A+B)$ in Theorem~\ref{th:DR} is automatically guaranteed, thanks to Proposition~\ref{prop:pde}\eqref{it:pde_i}. Note also that, as a simple linear boundary value problem, \eqref{eq:LP0} can be easily implemented using standard finite element method solvers. Following~\cite{adly2019decomposition}, we used the finite element library \emph{FreeFem}++ with a P1 finite element discretisation.

In our computational tests we set the parameter $\lambda=2$, which appears to be optimal for this problem. This empirical observation is in accordance with~\cite[Theorem~3.2]{aragon2019optimal}, which proves that the optimal rate of linear convergence of AAMR when it is applied to two subspaces is attained at $\kappa=\frac{\lambda}{2}=1$ (see also Remark~\ref{rem:AAMR}). The second pair of parameters we set was $\sigma_A=\sigma_B=0.25$. Observe that the behaviour of the iterative process~\eqref{eq:alg_DR} when $\sigma_A=\sigma_B$ is driven by the parameter $\gamma\sigma_A$, so in order to evaluate the variation of the algorithm's performance with respect to the parameters, one can either fix $\gamma$ or $\sigma_A$. We did not observe any apparent advantage of choosing $\sigma_A\neq\sigma_B$ in the current setting.

We compared the performance of S-DR for $\gamma\in\{0.1,0.2,\ldots,4.9,5.0\}$. Recall that for $\gamma=4$ the algorithm coincides with the one proposed by Adly--Bourdin (AB in short) in~\cite[Proposition~5]{adly2019decomposition}), see Remark~\ref{rem:AAMR}(ii). In our first experiment we used the data function $f(x,y)=xe^{-x^2-y^2}$ with $\Omega$ the open disk of centre $(0,0)$ and radius $\frac{3\pi}{2}$, which was tested in~\cite{adly2019decomposition}. We began by computing with AB an approximate solution $v_{AB}$ to $J_{A+B}(f)$ by running the algorithm for 10000 iterations with a mesh constructed by \emph{FreeFem}++ with 200 points in the boundary, see Figures~\ref{fig:solAdly_a} and \ref{fig:solAdly_b}. Then, we ran S-DR for each $\gamma\in\{0.1,0.2,\ldots,4.9,5\}$. The algorithm was stopped when the norm of the difference between $v_{AB}$ and the current iterate was smaller than a given precision of $10^{-p}$, with $p\in\{5,6,\ldots,10\}$. The results are summarised in Figure~\ref{fig:Adly}. A value of $\gamma$ between $0.4$ and $0.6$ appears to be optimal for all tested values of $p$. In particular, for $\gamma=0.5$ and $p=10$, S-DR was 8 times faster than AB.

\begin{figure}[!htbp]
\centering
  \begin{subfigure}[c]{0.48\textwidth}
     \includegraphics[height=.74\textwidth]{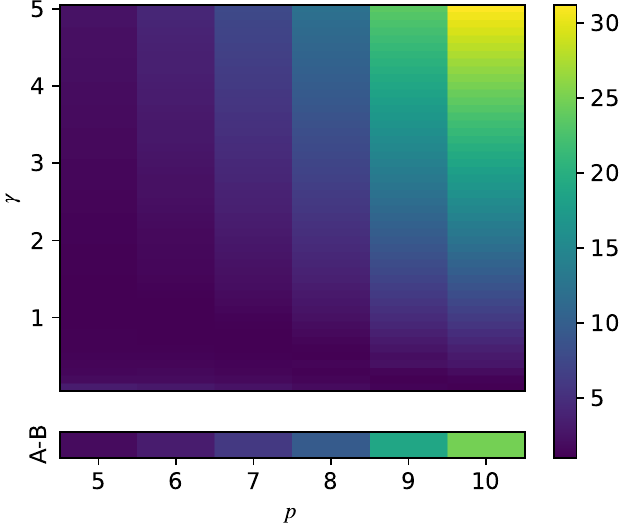}
  \end{subfigure}
  \begin{subfigure}[c]{0.48\textwidth}
     \includegraphics[height=.74\textwidth]{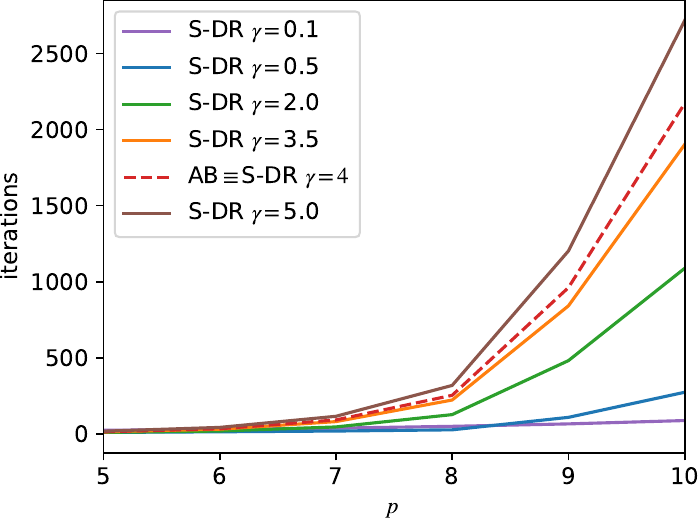}
  \end{subfigure}
\caption{(Left) For each precision $10^{-p}$, we computed the minimum number of iterations among AB and S-DR for each $\gamma\in \{0.1,0.2,\ldots,4.9,5\}$. The figure is coloured according to the ratio between the number of iterations of each of the methods and that minimum number. (Right) The number of iterations of AB and S-DR for different values of $\gamma$.}\label{fig:Adly}
\end{figure}

In our second experiment we used the function
\begin{equation}
f(x,y)=\left\{\begin{array}{ll}
-2\left((10y\pi-5y^2+1)\cos(x)^2-4y\pi+2y^2-1\right)\sin(x), & x\leq\pi,\\
(2\pi-y)y\cos(x)^2\sin(x)^3, & x>\pi;
\end{array}\right.\label{eq:f_Matt}
\end{equation}
with $\Omega=(0,2\pi)\times (0,2\pi)$. In this case, the unique solution of~\eqref{eq:PBP0} can be analytically computed and is given by $u(x,y):=(2\pi-y)y\sin(x)^3$, see Figures~\ref{fig:solAdly_c} and \ref{fig:solAdly_d}. Therefore, it is possible to test how good the approximate solution given by each of the algorithm settings is. In Figure~\ref{fig:Adly2} we show the result of running the algorithms for $10000$ iterations, again with a mesh with $200$ points in the boundary.

\begin{figure}[!htbp]
\centering
  \begin{subfigure}[c]{0.48\textwidth}
     \includegraphics[height=.75\textwidth]{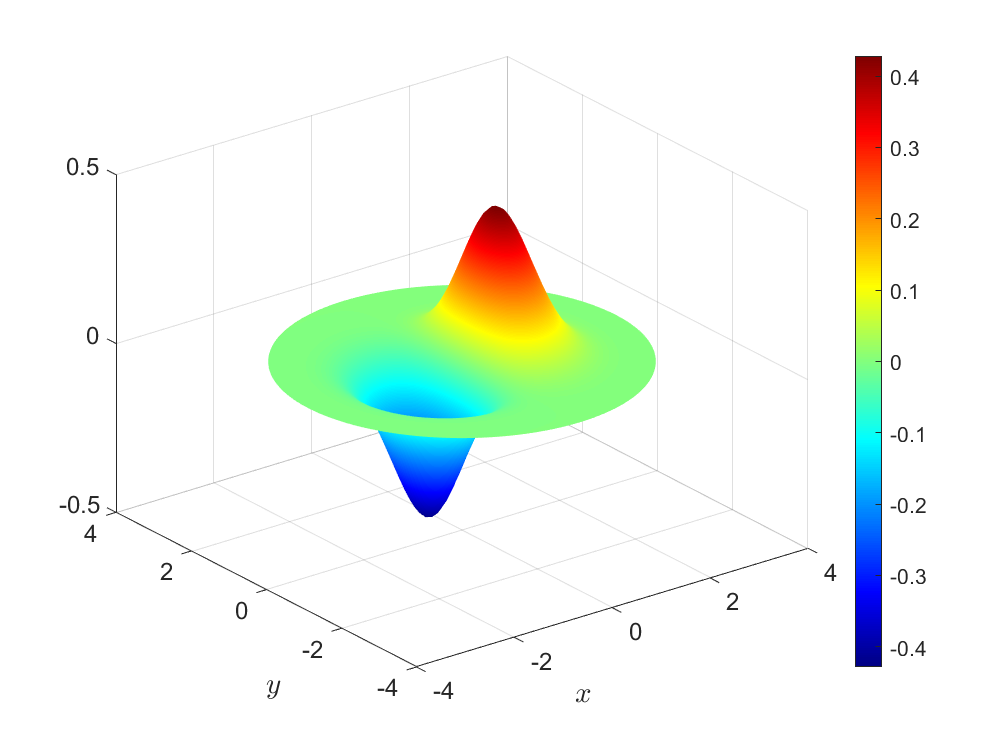}
     \subcaption{Plot of the data $f(x,y)=xe^{-x^2-y^2}$.\label{fig:solAdly_a}}
  \end{subfigure}\quad
  \begin{subfigure}[c]{0.48\textwidth}
     \includegraphics[height=.75\textwidth]{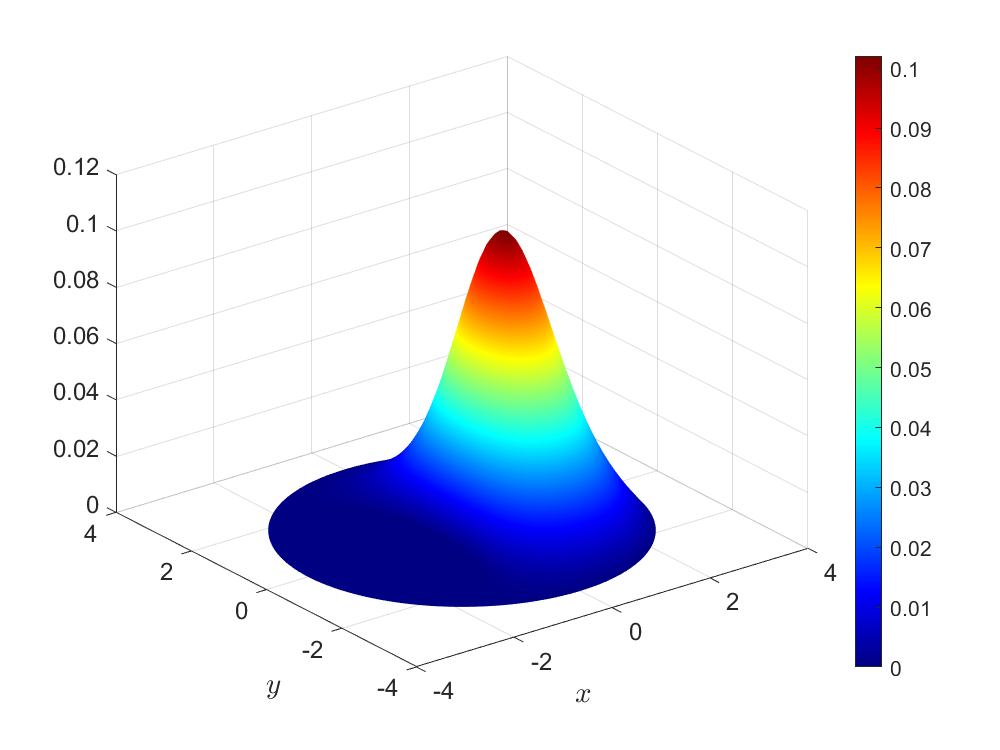}
     \subcaption{For $f(x,y)=xe^{-x^2-y^2}$, plot of the solution of~\eqref{eq:OP0} computed with AB.\label{fig:solAdly_b}}
  \end{subfigure}\\
  \begin{subfigure}[c]{0.48\textwidth}
     \includegraphics[height=.75\textwidth]{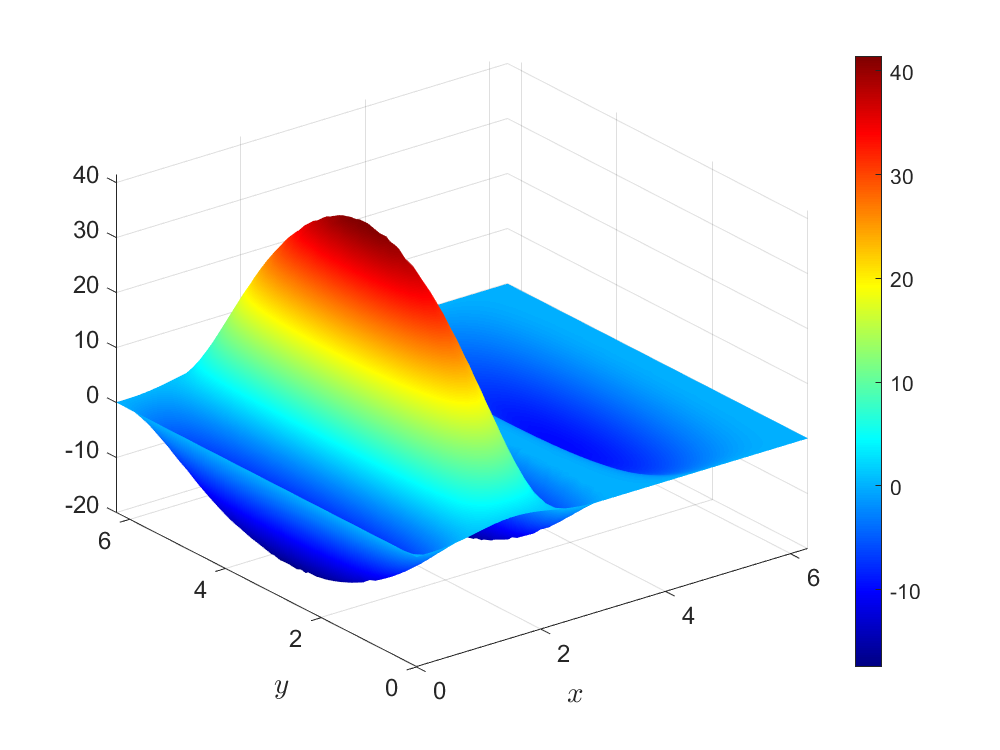}
     \subcaption{Plot of the data~\eqref{eq:f_Matt}.\label{fig:solAdly_c}}
  \end{subfigure}\quad
  \begin{subfigure}[c]{0.48\textwidth}
     \includegraphics[height=.75\textwidth]{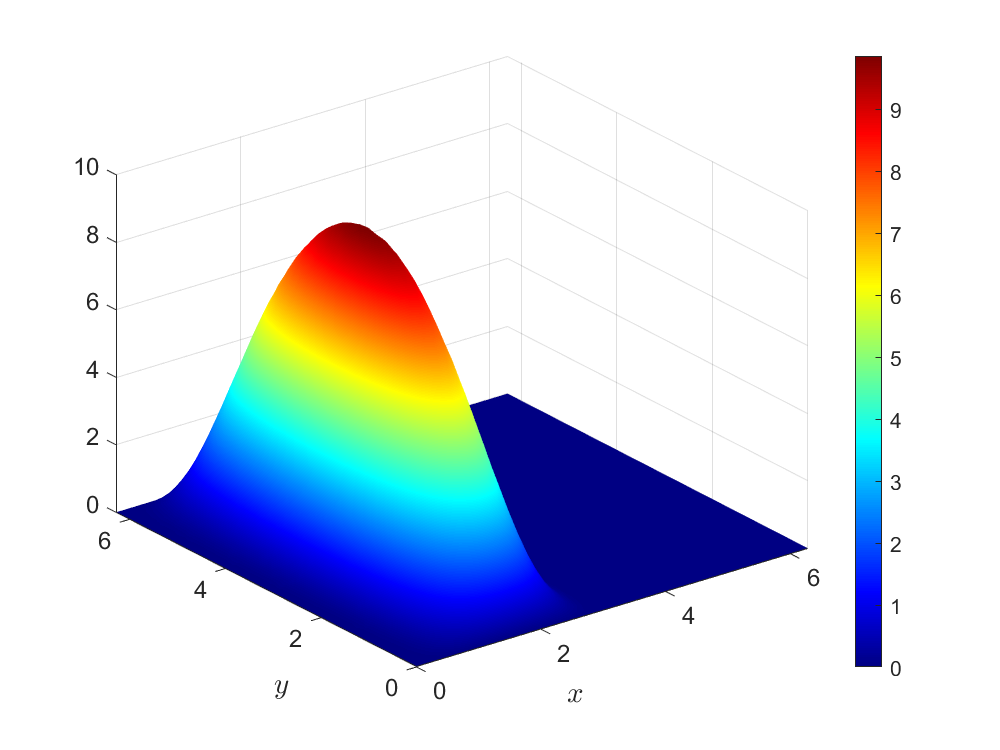}
     \subcaption{For the data~\eqref{eq:f_Matt}, plot of the solution  of~\eqref{eq:OP0} $v(x)=\max\{0,(2\pi-y)y\sin(x)^3\}$.\label{fig:solAdly_d}}
  \end{subfigure}
\caption{Representation of the data (left) and the solution (right) of two~\eqref{eq:OP0} problems.}\label{fig:solAdly}
\end{figure}
\begin{figure}[!htbp]
\centering
  \begin{subfigure}[c]{0.48\textwidth}
     \includegraphics[height=.7\textwidth]{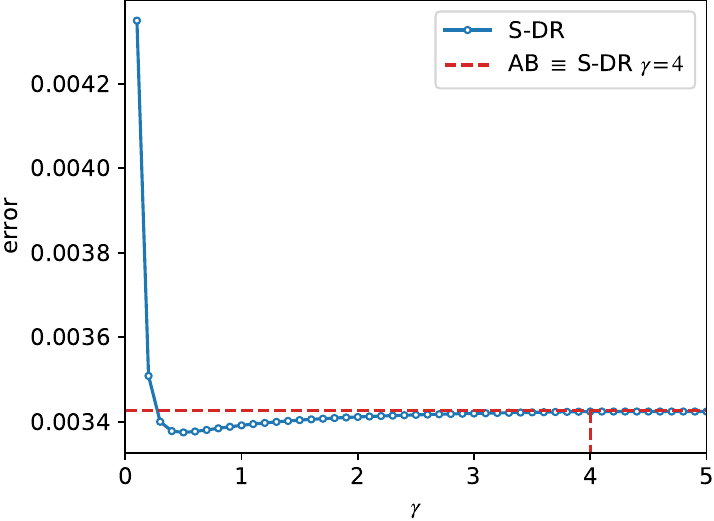}
  \end{subfigure}
  \begin{subfigure}[c]{0.48\textwidth}
     \includegraphics[height=.7\textwidth]{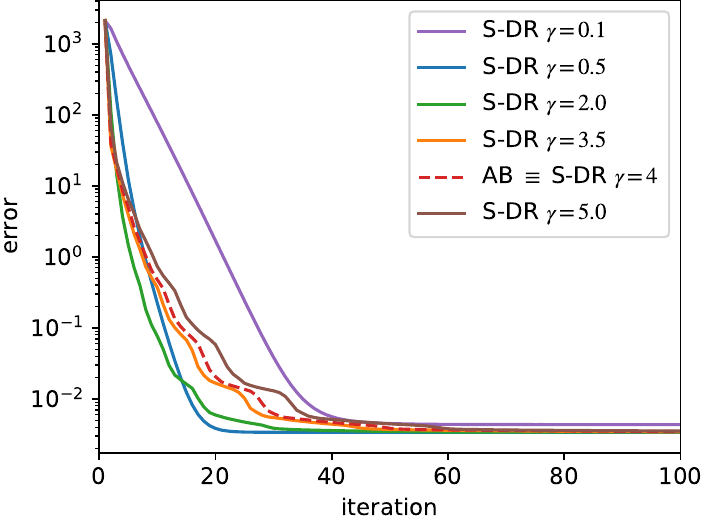}
  \end{subfigure}
\caption{(Left) The error in the solution of AB and S-DR with parameter $\gamma\in\{0.1,0.2,\ldots,4.9,5\}$ after running the algorithms for $10000$ iterations. (Right) The error (in log.\ scale) as a function of iterations for different values of $\gamma$.}\label{fig:Adly2}
\end{figure}

Therefore,  when $\sigma_A=\sigma_B$, the results in these experiments suggest that the value of $\gamma\sigma_A$ has a considerable effect on the performance of S-DR. If this value is chosen too small, the solution found may be more inaccurate, while a large value can slow down the algorithm. A value of $\gamma\sigma_A=0.125$ appeared to be optimal in terms of the number of iterations and the error for both of the function data considered.

\paragraph{{\small Acknowledgements}}{\small
\hspace{-2mm}The authors are thankful to Samir Adly for willingly sharing the FreeFem++ code from~\cite{adly2019decomposition}, and to the two anonymous referees whose comments helped to improve the paper. The authors would also like to thank Heinz Bauschke and Walaa Moursi for pointing out reference \cite{giselsson_tight} and providing us the proof of the case $\lambda=2$ in Theorem~\ref{th:DR}.

FJAA and RC were partially supported by the Ministry of Science, Innovation and Universities of Spain and the European Regional Development Fund (ERDF) of the European Commission, Grant PGC2018-097960-B-C22. MKT is supported in part by ARC grant DE200100063.}

\appendix

\section{Proof of Convergence for Ryu Splitting}\label{s:appendix}
In this appendix, we offer a direct proof for convergence of the method proposed in \cite[Section~4]{ryu2019uniqueness}. In doing so, we also extend the result to the infinite dimensional setting as well as establishing weak convergence of the shadow sequence, which is required to prove Theorem~\ref{th:ryu}.

Let $A,B,C\colon\Hilbert\setto\Hilbert$ be maximally monotone operators. Consider the problem
 $$ \text{find~}x\in\Hilbert\text{ such that }0\in (A+B+C)(x). $$
Note that $u\in\zer(A+B+C)$ if and only if there exist points $x,y\in\Hilbert$ such that
  $$ x-u\in \gamma A(u),\qquad y\in \gamma B(u),\qquad u-x-y\in \gamma C(u). $$
Using the definition of the resolvent, the latter is equivalent to
\begin{equation}\label{eq:uxy}
 u = J_{\gamma A}(x),\qquad u = J_{\gamma B}(u+y), \qquad u = J_{\gamma C}(2u-x-y).
\end{equation}
Consequently, we have $\zer(A+B+C)\neq\emptyset$ if and only if $\Omega\neq\emptyset$, where
$$ \Omega := \bigl\{(x,y)\in\Hilbert\times\Hilbert: J_{\gamma A}(x)=J_{\gamma B}(J_{\gamma A}(x)+y)= J_{\gamma C}(2J_{\gamma A}(x)-x-y)\bigr\}. $$

Recall that an operator $A\colon\Hilbert\setto\Hilbert$ is said to be \emph{uniformly monotone} with modulus $\phi\colon\R_+\to{[0,+\infty[}$ if $\phi$ is increasing, vanishes only at $0$ and
$$ \langle x-y,u-v\rangle \geq \phi(\|x-y\|)\quad\forall(x,u),(y,v)\in\gra A. $$

\begin{theorem}[Ryu splitting]\label{th:ryu appendix}
Let $A,B,C\colon\Hilbert\setto\Hilbert$ be maximally monotone operators with $\zer(A+B+C)\neq \emptyset$. Let $\gamma>0$ and $\lambda>0$. Given some initial points $x_0,y_0\in\Hilbert$, consider the sequences given by
\begin{equation}\label{eq:ryu splitting}
\left\{\begin{aligned}
    u_k &= J_{\gamma A}(x_k) \\
    v_k &= J_{\gamma B}(u_k+y_k) \\
    w_k &= J_{\gamma C}(u_k-x_k+v_k-y_k) \\
    x_{k+1} &= x_k + \lambda(w_k-u_k) \\
    y_{k+1} &= y_k + \lambda(w_k-v_k).
\end{aligned}\right.
\end{equation}
The following assertions hold.
\begin{enumerate}[(i)]
\item\label{it:ryu a i} If $\lambda\in{]0,1[}$, then $(x_k,y_k)\wto(x,y)\in\Omega$, $u_k\wto u, v_k\wto u$ and $w_k\wto u$ with
\begin{equation}\label{eq:shadow omega}
 u=J_A(x)=J_{\gamma B}(J_{\gamma A}(x)+y)= J_{\gamma C}(2J_{\gamma A}(x)-x-y)\in\zer(A+B+C).
\end{equation}
\item\label{it:ryu a ii} If $\lambda\in{]0,1]}$ and any of $A,B$ or $C$ is uniformly monotone, then $(u_k), (v_k)$ and $(w_k)$ converge strongly.
\end{enumerate}
\end{theorem}
\begin{proof}
\eqref{it:ryu a i}:~Let $(x,y)\in\Omega$ and denote $u:=J_{\gamma A}(x)$. Since $x-u\in\gamma A(u)$ and  $x_k-u_k\in \gamma A(u_k)$, monotonicity of $\gamma A$ implies
 \begin{equation}\label{eq:monoA}
 \begin{aligned}
  0 &\leq \langle (x-u)-(x_k-u_k),u-u_k \rangle \\
    &= \langle (x-u)-(x_k-u_k),u-w_k \rangle + \langle (x-u)-(x_k-u_k),w_k-u_k \rangle.
 \end{aligned}
 \end{equation}
Since $y\in \gamma B(u)$ and $u_k+y_k-v_k\in \gamma B(v_k)$, monotonicity of $\gamma B$ implies
\begin{equation}\label{eq:monoB}
\begin{aligned}
0 &\leq \langle y-(u_k+y_k-v_k),u-v_k\rangle \\
  &=   \langle y-(u_k+y_k-v_k),u-w_k\rangle + \langle y-(u_k+y_k-v_k),w_k-v_k\rangle.
\end{aligned}
\end{equation}
Since $u-x-y\in \gamma C(u)$ and $u_k-x_k+v_k-y_k-w_k\in \gamma C(w_k)$, monotonicity of $\gamma C$ implies
\begin{equation}\label{eq:monoC}
\begin{aligned}
 0 &\leq \langle(u-x-y)-(u_k-x_k+v_k-y_k-w_k),u-w_k\rangle \\
   &= \langle w_k-u_k,u-w_k\rangle - \langle(x-u)-(x_k-u_k),u-w_k\rangle
      -\langle y -(u_k+y_k-v_k),u-w_k\rangle.
\end{aligned}
\end{equation}
Summing together \eqref{eq:monoA}, \eqref{eq:monoB} and \eqref{eq:monoC} yields
\begin{equation}\label{eq:combined}
\begin{aligned}
0 &\leq  \langle (x-u)-(x_k-u_k),w_k-u_k \rangle + \langle y-(u_k+y_k-v_k),w_k-v_k\rangle + \langle w_k-u_k,u-w_k\rangle \\
  &= \langle x-x_k,w_k-u_k\rangle - \|w_k-u_k\|^2 + \langle y-y_k,w_k-v_k\rangle -\|w_k-v_k\|^2 + \langle w_k-u_k,w_k-v_k\rangle.
\end{aligned}
\end{equation}
The first and second terms in \eqref{eq:combined} can be expressed as
\begin{align*}
2\lambda \left( \langle x-x_k,w_k-u_k\rangle - \|w_k-u_k\|^2\right)
 &= 2\langle x-x_k,x_{k+1}-x_k\rangle - \frac{2}{\lambda}\|x_{k+1}-x_k\|^2 \\
 &= \|x_k-x\|^2 - \|x_{k+1}-x\|^2 - \left(\frac{2}{\lambda}-1\right)\|x_{k+1}-x_k\|^2.
\intertext{Similarly, the third and fourth terms in \eqref{eq:combined} can be written as}
2\lambda \left( \langle y-y_k,w_k-v_k\rangle -\|w_k-v_k\|^2 \right)
 &= \|y_k-y\|^2 - \|y_{k+1}-y\|^2 - \left(\frac{2}{\lambda}-1\right)\|y_{k+1}-y_k\|^2.
\end{align*}
The last term in \eqref{eq:combined} can be estimated as
$$ 2\lambda \langle w_k-u_k,w_k-v_k\rangle = \frac{2}{\lambda}\langle x_{k+1}-x_k,y_{k+1}-y_k \rangle \leq \frac{1}{\lambda}\|x_{k+1}-x_k\|^2 + \frac{1}{\lambda}\|y_{k+1}-y_k\|^2. $$
Altogether, we have
\begin{equation}
\|x_{k+1}-x\|^2
+\|y_{k+1}-y\|^2  + \left(\frac{1-\lambda}{\lambda}\right)\left(\|x_{k+1}-x_k\|^2+\|y_{k+1}-y_k\|^2\right)
\leq \|x_k-x\|^2 + \|y_k-y\|^2,\label{eq:ryu key}
\end{equation}
where we note that $\frac{1-\lambda}{\lambda}>0$ since $\lambda\in{]0,1[}$. For convenience, denote $z_k=(x_k,y_k)$ and $z=(x,y)$. Then \eqref{eq:ryu key} implies that $(z_k)$ is Fej\'er monotone with respect to $\Omega$ (so in particular, $(\|z_k-z\|)$ is nonincreasing and convergent, see e.g.~\cite[Proposition~5.4]{bauschke2017}), and that $z_{k+1}-z_k\to(0,0)$. Since resolvents are nonexpansive, it follows that $(u_k)$, $(v_k)$ and $(w_k)$ are bounded, and that $w_k-u_k\to 0$ and $w_k-v_k\to0$.

Let $\bar{z}=(\bar{x},\bar{y})$ be a weak sequential cluster point of the bounded sequence $(z_k)$. Then there exists a weak sequential cluster cluster point $\bar{u}$ of $(u_k)$ such that there exists a subsequence of $(z_{k},u_{k})$ which converges weakly to $(\bar{z},\bar{u})$. Now, from \eqref{eq:ryu splitting}, it follows that
$$ \colvec{u_k-w_k\\v_k-w_k\\u_k-w_k} \in T\colvec{x_k-u_k\\ u_k+y_k-v_k\\ w_k} := \left(\colvec{(\gamma A)^{-1}\\ (\gamma B)^{-1}\\ \gamma C}+\colvec{0&0&-\Id\\0&0&-\Id\\\Id&\Id&0}\right)\colvec{x_k-u_k\\ u_k+y_k-v_k\\ w_k} $$
where we note that the operator $T\colon\Hilbert^3\setto\Hilbert^3$ is maximally monotone as the sum of a maximally monotone operator and a skew-symmetric matrix (see, e.g., \cite[Example~20.35~\&~Corollary~25.5(i)]{bauschke2017}). Since the graph of a maximally monotone operator is sequentially closed in the weak-strong topology \cite[Proposition~20.38]{bauschke2017}, taking the limit along a subsequence of $(z_k,u_k)$ which converges weakly to $(\bar{z},\bar{u})$ (and noting that $w_k-u_k\to 0$ and $w_k-v_k\to0$, which  imply $u_k-v_k\to 0$ and $w_k\wto\bar{u}$) yields
\begin{equation}\label{eq:ryu limit}
\colvec{0\\0\\0} \in \left( \colvec{(\gamma A)^{-1}\\ (\gamma B)^{-1}\\ \gamma C}+\colvec{0&0&-\Id\\0&0&-\Id\\\Id&\Id&0}\right)\colvec{\bar{x}-\bar{u}\\ \bar{y}\\ \bar{u}} \implies \left\{\begin{array}{l}\bar{z}=(\bar{x},\bar{y})\in\Omega\\ \bar{u}=J_{\gamma A}(\bar{x}).\end{array} \right.
\end{equation}
In particular, this shows that every weak sequential cluster point of $(z_k)$ is contained in $\Omega$, and so \cite[Theorem~5.5]{bauschke2017} implies that $(z_k)$ converges weakly to a point $\bar{z}\in\Omega$. Then \eqref{eq:ryu limit} shows that $\bar{u}=J_{\gamma A}(\bar{x})$ is necessarily the unique cluster point of $(u_k)$, and hence $u_k\wto\bar{u}$. It follows that $v_k\wto\bar{u}$ and $w_k\wto\bar{u}$. The fact that $\bar{u}$ satisfies \eqref{eq:shadow omega} follows by the argued in after \eqref{eq:uxy}.

\eqref{it:ryu a ii}:~Suppose $A$ is uniformly monotone with modulus $\phi$. Then, using uniform monotoncity of $A$ in place of monotonicity in \eqref{eq:monoA}, yields the stronger inequality
$$ \gamma\phi\big(\|u_k-u\|\bigl) \leq \langle (x-u)-(x_k-u_k),u-w_k \rangle + \langle (x-u)-(x_k-u_k),w_k-u_k \rangle. $$
By propagating this inequality through the remainder of the proof, noting that $\lambda\in{]0,1]}$, \eqref{eq:ryu key} becomes
$$ \|x_{k+1}-x\|^2 +\|y_{k+1}-y\|^2 + 2\lambda\gamma\phi\bigl(\|u_k-u\|\bigr) \leq \|x_k-x\|^2 + \|y_k-y\|^2. $$
From this it follows that $\phi\bigl(\|u_k-u\|\bigr)\to0$ and hence that $u_k\to u$. When $B$ (resp.\ $C$) is uniformly monotone, the result follows by an analogous argument by modifying \eqref{eq:monoB} (resp.\ \eqref{eq:monoC}).
\end{proof}

\end{document}